\documentclass[12pt]{amsart}
\usepackage{amsmath,amsthm,amsfonts,amssymb,eucal}

\usepackage{color,bbm,interval,mathtools}

\renewcommand {\a}{ \alpha }
\renewcommand{\b}{\beta}

\newcommand{\g}{\gamma}
\newcommand{\G}{\Gamma}

\renewcommand{\d}{\delta}
\newcommand{\s}{\sigma}
\renewcommand{\l}{\lambda}
\renewcommand{\L}{\Lambda}
\newcommand{\z}{\zeta}
\renewcommand{\t}{\theta}

\newcommand{\p}{\partial}
\newcommand{\om}{\omega}
\newcommand{\Om}{\Omega}

\newcommand{\R}{ \mathbb R}

\newcommand{\SfG}{{\sf{G}}}
\newcommand{\sg}{{\sf{g}}}

\newcommand {\GS}{\mathfrak S}

\newcommand {\BM}{\mathbf M}

\newcommand {\BS}{\mathbf S}

\newcommand {\BO}{\mathbf O}

\newcommand {\bk}{\mathbf k}

\newcommand {\bm}{\mathbf m}

\newcommand {\bn}{\mathbf n}

\newcommand{\SA}{{\sf{A}}}
\newcommand{\SP}{{\sf{P}}}

\newcommand{\sv}{{\sf{v}}}

\newcommand {\bnu}{\boldsymbol\nu}
\newcommand {\bom}{\boldsymbol\om}

\newcommand{\btau}{\boldsymbol\tau}

\newcommand{\lu}{\langle}
\newcommand{\ru}{\rangle}

\newcommand{\CP}{\mathcal P}

\newcommand{\CA}{\mathcal A}

\newcommand{\CC}{\mathcal C}
\newcommand{\CS}{\mathcal S}

\newcommand{\plainW}[2]{\textup{{\textsf{W}}}^{#1, #2}}

\newcommand{\plainC}[1]{\textup{{\textsf{C}}}^{#1}}

\newcommand{\plainL}[1]{\textup{{\textsf{L}}}^{#1}}
\newcommand{\plainl}[1]{\textup{{\textsf{l}}}^{#1}}

\newcommand{\1}
{{\,\vrule depth3pt height9pt}{\vrule depth3pt height9pt}
{\vrule depth3pt height9pt}{\vrule depth3pt height9pt}\,}

\DeclareMathOperator{\op}{{Op}}

\newcommand{\id}{{\rm {\mathbf 1}}}

\newcommand{\Z}{\mathbb Z}

\newtheorem{thm}{Theorem}[section]
\newtheorem{cor}[thm]{Corollary}
\newtheorem{lem}[thm]{Lemma}
\newtheorem{prop}[thm]{Proposition}

\theoremstyle{definition}

\newtheorem{rem}[thm]{Remark}

\numberwithin{equation}{section}

\begin{document}
\hoffset -4pc

\title
[Spectral asymptotics ]
{{Spectral asymptotics of pseudodifferential operators with discontinuous symbols}}
\author[A. Derkach]{A. Derkach}
\author[A. V. Sobolev]{A. V. Sobolev}
 \address{Department of Mathematics\\ University College London\\
Gower Street\\ London\\ WC1E 6BT UK}
\email{alexey.derkach.21@alumni.ucl.ac.uk}

 \address{Department of Mathematics\\ University College London\\
Gower Street\\ London\\ WC1E 6BT UK}
\email{a.sobolev@ucl.ac.uk}

\begin{abstract}
We study discrete spectrum of self-adjoint 
Weyl pseudodifferential operators with discontinuous 
symbols of the form $\id_\Om \phi$ where $\id_\Om$ is 
the indicator of a domain in $\Om\subset\R^2$, and 
$\phi\in\plainC\infty_0(\R^2)$ is a real-valued function.  
It was known that in general, the singular values $s_k$ of such an operator satisfy the bound 
$s_k = O(k^{-3/4})$, $k = 1, 2, \dots$. We show that if $\Om$ is a polygon, 
the singular values decrease as $O(k^{-1}\log k)$.  
In the case where $\Om$ is a sector, we obtain an asymptotic formula 
which confirms the sharpness of the above bound. 

Our main technical tool is the reduction to another symbol that we call \textit{dual}, 
which is automatically smooth. To analyse the dual symbol 
we find new bounds for singular values of 
pseudodifferential operators with smooth symbols
 in $\plainL2(\R^d)$ for arbitrary dimension $d\ge 1$.  
\end{abstract}

\maketitle

\section{Introduction}

\subsection{Background} 
One of the challenging and important questions of spectral theory 
of pseudodifferential operators is the study of operators with discontinuous symbols. 
In this paper we study pseudodifferential operators on $\plainL2(\R)$ defined on 
functions $u$ from the Schwartz class by the formula
\begin{align}\label{eq:aom}
(\SA_\Om[\phi]u)(x) = &\ \op^{\rm w}(\id_\Om \phi)\notag\\ 
:= &\ \frac{1}{2\pi} \, \iint_{\R^2}\, e^{i(x-y)\xi}\id_\Om
\bigg(\frac{x+y}{2}, \xi\bigg)\, 
\phi\bigg(\frac{x+y}{2}, \xi\bigg)\, u(y)\, dy d\xi. 
\end{align}
Here $\id_\Om$ is the indicator of $\Om\subset \R^2$, and 
$\phi\in \plainC\infty_0(\R^2)$. If $\Om$ is bounded then an arbitrary 
$\phi\in \plainC\infty(\R^2)$  
is allowed. In any case, the operator $\SA_\Om = \SA_\Om[\phi]$ is compact, 
even Hilbert-Schmidt, see Lemma \ref{lem:hs}. 
We assume that $\phi$ is real-valued, and hence $\SA_\Om[\phi]$ is self-adjoint. Due 
to the dependence on $(x+y)/2$ (and not, e.g. on $x$ or $y$ alone), the function 
$\phi_\Om(\btau) := \id_\Om(\btau) \phi(\btau)$, $\btau = (x, \xi)$, is called 
the \textit{Weyl symbol} of the operator $\SA_\Om$.  
Denote by $\l_k^{(+)}$ and $-\l_k^{(-)}$, $k = 1, 2, \dots,$ 
the positive and negative eigenvalues of $\SA_\Om$. 
Each sequence $\l^{(+)}_k$, $\l_k^{(-)}$ 
is enumerated  in non-increasing order counting multiplicity.  Sometimes we write 
$\l_k^{(\pm)}(\SA_\Om)$ or $\l_k^{(\pm)}(\SA_\Om[\phi])$ to indicate the dependence on the operator.
We are interested in the decay rate of the eigenvalues as $k\to\infty$. 
This spectral problem naturally embeds into the 
context of \textit{time-frequency analysis}, see \cite{Gro2001}.
The phase space $\R^2$ is then interpreted as the time-frequency plane, and the primary 
problem of the analysis is  
to find a function (signal) that 
is concentrated in both time and frequency within a fixed domain $\Om\subset\R^2$. 
More precisely, one is looking for functions $u\in \plainL2(\R), \|u\|_{\plainL2}=1$,  
that deliver consecutive maxima  
of the quadratic form $\big(\SA_\Om[\phi] u, u\big)$, see \cite{Flandrin1988}. In spectral terms, 
this means that one seeks eigenfunctions and eigenvalues of the operator $\SA_\Om[\phi]$. 
For piece-wise smooth bounded 
domains $\Om$ and $\phi = 1$, it was shown in \cite{Ramanathan1993} that 
\begin{align}\label{eq:nottrace}
\sum_k\, \l_k^{(+)} + \sum_k \, \l_k^{(-)} = \infty,
\end{align}
so that the operator $\SA_\Om[1]$ is not trace class, and in general, 
\begin{align}\label{eq:34}
\l_k^{(\pm)} = O(k^{-\frac{3}{4}}).
\end{align} 
An alternative proof of \eqref{eq:34} was given in \cite[Corollary 3.2.4]{Derkach2024}. 
A calculation for annular domains in \cite{Ramanathan1993} demonstrates that 
this bound is sharp.  
In the case where $\Om$ is an ellipse, 
an explicit formula for the eigenvalues was found in 
 \cite{Flandrin1988}, 
see also \cite[Ch. 3]{Lerner_2024}. 

\subsection{Main result}
In the present paper we find that for polygonal domains (not necessarily convex) 
the eigenvalues $\l_k^{(\pm)}$ decay faster than in \eqref{eq:34}.  
Our main result gives an asymptotic formula for $\l_k^{\pm)}$ in the case where 
$\Om\subset \R^2$ is a sector.
We always assume that $\Om$ is given by 
\begin{align*}
\Om_{\t} = \{(x, \xi)\in\R^2: 0<\arg(x+i\xi) < \t\},\quad\textup{with some opening angle}\quad 
\t\in (0, 2\pi),
\end{align*} 
see Remark \ref{rem:sympl}\eqref{item:rot} for justification. 
The next theorem is the main result of the paper. 

\begin{thm}\label{thm:main} 
Let $\Om = \Om_\t\subset \R^2$ be a sector with the opening angle 
$\t\in (0, \pi)$ or $\t\in (\pi, 2\pi)$, and let  $\phi\in\plainC\infty_0(\R^2)$ be real-valued. 
Then 
\begin{align}\label{eq:main}
\lim_{k\to\infty}\,\l_k^{(\pm)}(\SA_\Om[\phi])\, 
\frac{k}{\log k} = \frac{|\phi(0, 0)|}{4\pi^2}.
\end{align}
If $\phi(0, 0) = 0$, then $\l^{(\pm)}\big(\SA_{\Om_\t}[\phi]\big) = O(k^{-1})$. 
\end{thm} 
 
\begin{rem}\label{rem:sympl}
\begin{enumerate}
\item 
The asymptotics \eqref{eq:main} are in agreement with \eqref{eq:nottrace}.
\item 
The angle $\t = \pi$ is excluded from the statement, since 
the sector $\Om_{\pi}$ is a half-plane, and hence by the second part of the theorem,  
$\l^{(\pm)}\big(\SA_{\Om_\pi}[\phi]\big) = O(k^{-1})$. 
\item\label{item:rot} 
The asymptotic formula \eqref{eq:main} easily generalizes to an arbitrary sector 
$\L_\t\subset\R^2$ with the opening angle $\t$. Indeed, $\L_\t$ can be represented as 
\begin{align*}
\L_\t = \BO(\Om_\t + \bom)
\end{align*}
with some vector $\bom\in\R^2$ and some rotation $\BO$. 
The associated metaplectic transformation 
provides unitary equivalence of $\SA_{\Om_\t}[\phi]$ and $\SA_{\L_\t}[\tilde\phi]$ 
where $\tilde\phi(\btau) = \phi(\BO^{-1}\btau - \bom), \btau = (x, \xi)$, 
see \cite[Theorem 4.51] 
{Folland_1989} and \cite[Theorem 1.2.20 and Lemma A.2.4]{Lerner_2024}. 

\item\label{item:sympl} 
The formula \eqref{eq:main} does not depend on the angle $\t$. This is not surprising 
at least for 
$\t\in (0, \pi)$ since in this case 
the sector $\Om_\t$ is symplecticaly equivalent to a quarter-plane, 
see 
\cite[P. 114]{Lerner_2024}. Precisely, the symplectic matrix
\begin{align*}
\BM_\t = 
\begin{pmatrix}
1 & - \cot\t\\
0 & 1
\end{pmatrix}
\end{align*}
acts as follows: $\BM_\t \Om_{\t} = \Om_{\pi/2}$. 
As in the previous remark, by an appropriate 
 metaplectic transformation the operators 
$\SA_{\Om_\t}[\phi]$ and $\SA_{\Om_{\pi/2}}[\tilde \phi]$, 
$\tilde \phi(\btau) = \phi(\BM_\t^{-1} \btau)$, are unitarily equivalent. Thus it suffices to 
prove \eqref{eq:main} for $\t = \pi/2$.

Discussion of the case $\t\in (\pi, 2\pi)$ is deferred until Section \ref{subsect:reduction}.
\end{enumerate}
\end{rem} 

The asymptotics \eqref{eq:main} lead to the following bound. 

\begin{cor}\label{cor:mainbound}
Let $\Om\in \R^2$ be a polygon (not necessarily convex), and let $\phi\in \plainC\infty_0(\R^2)$. 
Then  
\begin{align}\label{eq:upb}
\l_{k}^{(\pm)}(\SA_\Om[\phi]) = O(k^{-1}\log k),\quad k\to\infty.
\end{align}
\end{cor} 

Indeed, introducing a suitable partition of unity, using the Ky Fan inequalities 
from Sect. \ref{sect:compact} and Remark 
\ref{rem:sympl}\eqref{item:rot}, Theorem \ref{thm:main} implies \eqref{eq:upb}.

Theorem \ref{thm:main}  will be proved in Sect. \ref{sect:proofs}.

Although it is not directly relevant to the main results, 
it is instructive to compare spectral problem for the operator $\SA_\Om$ with 
another model adopted in time-frequency analysis. Instead of $\SA_\Om$ consider 
the operator 
\begin{align*}
(T_{\L, \G} u)(x) = \frac{1}{2\pi}\id_\L(x)\,\iint_{\R^2} \, e^{i(x-y)\xi} 
\id_\G(\xi) \id_\L(y) u(y)\, dy d\xi,
\end{align*}
 where $\L$ and $\G$ are bounded intervals of the real line, and 
 $\id_\L$ and $\id_\G$ are their 
indicator functions. A detailed study of the  
eigenfunctions and eigenvalues of $T_{\L, \G}$ (all of which are positive)  
was conducted by Landau, Pollack  and Slepian in a series of papers 
starting with \cite{SlePol1961}. Further analysis of the problem 
in \cite[Corollary 3]{BonKar2015} showed that $\l_k^{(+)}$  
decay superexponentially as $k\to\infty$, cf. \eqref{eq:main}. 
We also point out another difference between $\SA_\Om$ and $T_{\L, \G}$. Since $T_{\L, \G}$ 
is a product of three orthogonal projections, we always have $T_{\L, \G}\le 1$. 
As shown by Flandrin in \cite{Flandrin1988}, if $\Om$ is an ellipse, then 
$A_{\Om}[1]\le 1$ as well. In the same paper Flandrin conjectured that 
this bound should hold for 
any convex domain $\Om$. This conjecture was disproved  in \cite{DeDuLe2020} and  
\cite[Ch. 5]{Lerner_2024},  presenting 
several examples of convex domains, bounded and unbounded, 
for which supremum of the spectrum of $A_\Om[1]$ is strictly greater than $1$. 
One such example 
is given by the quarter-plane 
$\Om_{\pi/2}$. 

\subsection{The method} 
Our proof has three ingredients. The first two concern pseudodifferential 
operators on $\plainL2(\R^d)$ with arbitrary $d\ge 1$, and they can be of independent interest. 

At the heart of our approach is the notion of a 
\textit{dual} symbol that we introduce in \eqref{eq:dual}. For any square-integrable 
symbol $a(x, \xi)$ its dual $a^*(x, \xi)$ is defined (modulo some scaling)  
as its symplectic Fourier transform, see e.g. \cite[Prologue]{Folland_1989} 
for the definition of the simplectic Fourier transform. 
One useful feature of the transformation $\op^{\rm w}(a) \mapsto\op^{\rm w}(a^*)$ 
is that it preserves singular values.  
Furthermore, the utility of the dual symbol is especially evident when the original symbol $a$ 
is discontinuous but decays rapidly at infinity, so that $a^*$ is smooth, which 
puts at our disposal 
a vast arsenal of standards methods. This observation 
plays a key role in our analysis of the operator \eqref{eq:aom}.

Another critical ingredient of our method is a new bound for 
singular values of pseudodifferential operators 
on $\plainL2(\R^d)$ 
with smooth symbols. Very loosely speaking 
we show that under some decay conditions on the symbol $a = a(\btau), \btau = (x, \xi),$ 
the spectral counting function 
of the compact operator 
$A = \op^{\rm w}(a)$ 
is controlled by the phase space volume of the ``classically allowed" region, 
see Theorem \ref{thm:limsup}.
%
%
For operators whose symbols 
have product structure, i.e. $a(x, \xi) = f(x) g(\xi)$, such estimates 
were first derived by Cwikel \cite{Cwikel1977}, and then complemented by a result of Simon, 
see \cite[Theorem 4.6]{Simon2005}. 
We stress that we 
do not assume any separation of variables in the symbol. The proofs rely on the 
bounds in the Schatten-von Neumann classes, obtained by the second author 
in \cite{Sobolev2014}, and on the ideas from \cite[Ch. 4]{Simon2005}. 
 
The final ingredient is the Birman-Schwinger principle 
which links spectral counting function of the operator 
$\SA_\Om$ with that of a suitable Schr\"odinger operator on $\plainL2(\R)$. 
Schr\"odinger operators 
have been very well studied in the literature, with a large number 
of results on their 
discrete spectrum. We are relying on the asymptotic formula found in \cite{Sob_Sol_2002} 
which allows us to obtain \eqref{eq:main}.

\subsection{Plan of the paper} 
 Sect. \ref{sect:prelim} contains some basic information about pseudodifferential operators and 
 introduces the notion of a dual symbol. 
 Here we also calculate the dual symbol for the operator \eqref{eq:aom}. 
 In Sect. \ref{sect:compact} we collect some known facts 
on classes of compact operators with a specific decay rate   
of their singular numbers $s_k, k=1, 2, \dots$. 
This includes the class with the rate 
$s_k \sim k^{-1} \log k$, as in \eqref{eq:main}. Sect. \ref{sect:bounds} is of 
central importance: here we derive phase space-type bounds 
for pseudodifferential operators with smooth symbols, see Theorem \ref{thm:limsup}.
Then in Sect. \ref{sect:examples} we give two examples 
relevant to our study of operators with discontinuous symbols. One example in particular, 
provides a multi-dimensional generalization of 
the bound \eqref{eq:34} by applying Theorem \ref{thm:limsup} to the smooth dual symbol of the 
indicator $\id_\Om$ of a domain $\Om\subset\R^{2d}$, see Proposition \ref{prop:curved}. 
The second example prepares the ground for the proof of Theorem \ref{thm:main}.  
In Sect. \ref{sect:kntow} using the method of Sect. 
\ref{sect:bounds}, 
we show that changing the Weyl symbol $a\big((x+y)/2, \xi\big)$ to the 
Kohn-Nirenberg symbol $a(x, \xi)$ does not affect the spectral asymptotics. In Sect. \ref{sect:proofs} all the ingredients are put together to complete the proof of Theorem \ref{thm:main}, with the help of the Birman-Schwinger principle.
 
\subsection{Notational conventions} By 
 $\btau = (x, \xi)\in \R^{2d}$ we denote the pair of standard canonical $d$-dimensional 
 variables. 
 
\textit{Fourier transforms.} 
For a function $f(x, \xi)$ on $\plainL2(\R^{2d})$ introduce the notation
\begin{align*}
({\mathcal F}_1 f) (\eta, \xi) = &\ \, 
\frac{1}{(2\pi)^{\frac{d}{2}}}\int_{\R^d} e^{-i\eta\cdot x}\, f(x, \xi)\, dx,\\
 ({\mathcal F}_2 f) (x, y) = &\ \, 
\frac{1}{(2\pi)^{\frac{d}{2}}}\int_{\R^d} e^{-i y\cdot\xi}\, f(x, \xi)\, d\xi,
 \end{align*}
and  
\begin{align*}
(\mathcal F f)(\eta, y) 
:= ({\mathcal F}_1\,{\mathcal F}_2 f) (\eta, y) = \,  
\frac{1}{(2\pi)^d}
\iint\limits_{\R^{2d}} e^{-i y\cdot\xi - i \eta\cdot x}\, f(x, \xi)\, d\xi\, dx.
\end{align*}
Integrals without indication of the domain are assumed to be taken over the entire Euclidean 
space $\R^d$, or $\R^{2d}$ depending on the context. 

\textit{Derivatives and function spaces.} 
We use the generally accepted notation for partial derivatives: for  
$x = (x_1, x_2, \dots, x_d)\in \R^d$ and $m = (m_1, m_2, \dots, m_d)\in \mathbb N_0^{d}$, 
$\mathbb N_0 = \mathbb N \cup\{0\}$,  we write 
$\p_x^m f = \p_{x_1}^{m_1}\p_{x_2}^{m_2}\cdots \p_{x_d}^{m_d} f$, and 
\begin{align*}
|\nabla_x^k f | = \sum_{|m|=k} \, |\p_x^m f|,\ k = 0, 1, \dots. 
\end{align*}
The symbols $\plainW{l}{p}$, $l = 0, 1, \dots$, $p\ge 1$, stand for the standard  
Sobolev spaces, see e.g. \cite[Ch. 5]{Evans1998}.

\textit{Bounds.} 
For two non-negative numbers (or functions) 
$X$ and $Y$ depending on some parameters, 
we write $X\lesssim Y$ (or $Y\gtrsim X$) if $X\le C Y$ with 
some positive constant $C$ independent of those parameters. 
If $X\lesssim Y\lesssim X$, then we write $X\asymp Y$. 
 To avoid confusion we may comment on the nature of 
(implicit) constants in the bounds. 

\section{Pseudodifferential operators: preliminaries}\label{sect:prelim}

In this section we  consider pseudodifferential operators on 
$\plainL2(\R^d)$ for arbitrary dimension $d\ge 1$.

\subsection{Generalities}
We shall encounter several types of pseudodifferential operators. 

For a function (\textit{symbol}) $a = a(x, \xi)$, $x, \xi\in\R^d$ and 
arbitrary $t\in [0, 1]$, we define 
the $t$-pseudodifferential operator in the standard way:
\begin{align}\label{eq:opt}
(\op^{(t)}(a) u) (x) = \frac{1}{(2\pi)^d} \, \iint e^{i(x-y)\cdot\xi}
a((1-t)x + ty, \xi) u(y) \, dy d\xi,
\end{align} 
 where $u$ is in the Schwartz class $\CS(\R^d)$. We shall be  working mostly with $t=0$ and $t=1/2$, 
 in which case we use the standard notation 
 \begin{align*}
 \op^{\rm l}(a) := \op^{(0)}(a),\quad  \op^{\rm w}(a) := \op^{(1/2)}(a).
 \end{align*}
Recall that the symbol $a(x, \xi)$ of 
$\op^{\rm l}(a)$ is usually called \textit{the Kohn-Nirenberg symbol},  
and the symbol $a((x+y)/2, \xi)$ of $\op^{\rm w}(a)$ is called \textit{the Weyl symbol}.
We also need a more general type of a pseudodifferential operator. 
Let $p = p(x, y, \xi), x,y,\xi\in\R^d$, be a smooth function. 
Recall the standard notation for the pseudodifferential operator 
with \textit{amplitude} $p$:
\begin{align*}
(\op^{\rm a}(p) u)(x) = \frac{1}{(2\pi)^d} \, \iint e^{i(x-y)\cdot\xi}
p(x, y, \xi) u(y) \, dy d\xi.
\end{align*} 
Our assumptions on 
symbols $a$ and amplitudes $p$ will ensure that all the integrals above 
are convergent. In particular, we always assume that 
$a\in\plainL2(\R^{2d})$, in which case it is well known that 
the operators $\op^{(t)}(a)$ are Hilbert-Schmidt. We give the proof for completeness. 
Below the class of Hilbert-Schmidt operators is denoted by $\BS_2$. 

\begin{lem} \label{lem:hs}
If $a\in \plainL2(\R^{2d})$, then for all $t\in [0, 1]$ we have 
$\op^{(t)}(a)\in \BS_2$ and 
\begin{align}\label{eq:hs}
\|\op^{(t)}(a)\|_{\BS_2}^2 
= \frac{1}{(2\pi)^d}  \int_{\R^{2d}} |a(\btau)|^2 \, d\btau,\quad \btau = (x, \xi).
\end{align}
Let 
\begin{align}\label{eq:inta}
\SA_{\tiny{\rm int}}(s) = \int_0^s \op^{(t)}(a) \, dt,
\end{align}
with some $s>0$. 
Then $A_{\tiny{\rm int}}(s)\in \BS_2$ and 
\begin{align}\label{eq:int}
\|\SA_{\tiny{\rm int}}(s)\|_{\BS_2}^2 \le \frac{s^2}{(2\pi)^d}  \int_{\R^{2d}} |a(\btau)|^2 \, d\btau.
\end{align}
\end{lem}

\begin{proof}
The kernel of $\SA:=\op^{(t)}(a)$ is 
\begin{align*}
\frac{1}{(2\pi)^{\frac{d}{2}}}
(\mathcal F_2 a)((1-t)x+ty, y-x).
\end{align*}
Therefore 
\begin{align*}
\|\SA\|_{\BS_2}^2 = \frac{1}{(2\pi)^d} 
\iint |(\mathcal F_2 a)((1-t)x+ty, y-x)|^2 \, dx dy.
\end{align*}
Changing the variables to 
\begin{align*}
(1-t)x+ty = w,\quad y-x = z, 
\end{align*}
we obtain 
\begin{align*}
\|\SA\|_{\BS_2}^2 = \frac{1}{(2\pi)^d} 
\iint |(\mathcal F_2 a)(w, z)|^2 \, dw dz. 
\end{align*}
By Plancherel's identity, the last double integral coincides with the norm  
$\|a\|_{\plainL2}^2$, as claimed.  

The bound \eqref{eq:int} immediately follows from the estimates
\begin{align*}
\|\SA_{\tiny{\rm int}}(s)\|_{\BS_2}^2 
\le \bigg[ \int_0^s\,  \|\op^{(t)}(a)\|_{\BS_2} \, dt\bigg]^2
\le s \int_0^s\,  \|\op^{(t)}(a)\|_{\BS_2}^2 \, dt,
\end{align*}
and from \eqref{eq:hs}.
\end{proof}

\subsection{Reduction to a "dual" symbol}  
Let $a\in\plainL2(\R^{2d})$ be an arbitrary complex-valued function. 
Instead of the operator $\op^{\rm w}(a)$ one can study 
a pseudodifferential operator with a \textit{dual symbol}, which is defined as 
\begin{align}\label{eq:dual}
a^*(x, \xi) = 2^d\, (\mathcal F a)(2\xi, -2x).
\end{align}
The next theorem links $\op^{\rm w}(a)$ and $\op^{\rm w}(a^*)$ with one another. 
 
\begin{thm}\label{thm:red}
Let $a\in\plainL2(\R^{2d})$, and let 
$U:\plainL2(\R^d)\mapsto \plainL2(\R^d)$ be the unitary operator 
of reflection, i.e. $(U f)(x) = f(-x)$. Then  
$\op^{\rm w}(a) U = \op^{\rm w}(a^*)$.
\end{thm}

\begin{proof} Due to Lemma \ref{lem:hs} we may assume that 
$a\in \CS(\R^{2d})$. Introduce the new variables
\begin{align*}
z = \frac{x+y}{2},\quad v = \frac{x-y}{2}. 
\end{align*}
Thus, rewriting $\SA = \op^{\rm w}(a)$ as an integral operator we get 
\begin{align*}
(\SA f)(x) = &\ \int \CA(x, y)\, f(y)\, dy,\quad \CA(x, y) = 
\frac{1}{(2\pi)^d} \, \int \, e^{2 iv\cdot\xi} \, a( z, \xi)\, d\xi,
\end{align*}
Rewrite the kernel: 
\begin{align*}
\CA(x, y) = &\ \frac{1}{(2\pi)^{\frac{d}{2}}} (\mathcal F_2 a) (z, -2 v)\\
= &\ \frac{1}{(2\pi)^d} 
\int e^{i\eta\cdot z}(\mathcal F_1 \mathcal F_2 a) (\eta, - 2 v) \, 
d\eta.
\end{align*}
Therefore, 
\begin{align*}
\CA(x, y) = &\ 
\frac{1}{(2\pi)^d} 
\int e^{i(x+y)\cdot\frac{\eta}{2}}(\mathcal F a) 
(\eta, -(x-y)) \, d\eta\\
=  &\ \frac{2^d}{(2\pi)^d} 
\int e^{i(x+y)\cdot\eta}(\mathcal F a) 
( 2 \eta, - (x-y)) \, d\eta.
\end{align*} 
Consequently,
\begin{align*}
(\SA U\,f)(x) =  \frac{2^d}{(2\pi)^d} 
\iint e^{i(x-y)\cdot\eta}(\mathcal F a) 
( 2 \eta, - (x+y))  f(y)   \, dy  d\eta
= \big(\op^{\rm w}(a^*) f\big)(x),
\end{align*}
as claimed. 
\end{proof}

%
Note that the 
dual symbol can be rewritten as 
$a^*(\btau) = 2^d \big({\mathcal F}_{\tiny{\rm symp}}a\big)(2\btau), \btau = (x, \xi)$, 
with ${\mathcal F}_{\tiny{\rm symp}}$ being the so-called \textit{symplectic Fourier transform}, 
see e.g. \cite[Prologue]{Folland_1989}. We do not use this fact in what follows.
%

\begin{rem}\label{rem:dual} 
Theorem \ref{thm:red} can be slighly generalized in the following way. 
Let $a$ be as in the theorem, and let $g = g(x)$ be a bounded function. Denote 
\begin{align*}
p(x, y, \xi) = a\bigg(\frac{x+y}{2}, \xi\bigg)\, g(x-y).
\end{align*}
Then $\op^{\rm a}(p)U = \op^{\rm w}(b)$ with
\begin{align*}
b(x, \xi) = 2^d \big({\mathcal F} a\big) (2\xi, -2x) \, g(-2x). 
\end{align*} 
The proof requires only minor modifications. 

\end{rem}

Theorem \ref{thm:red} will be our main tool when proving Theorem \ref{thm:main}. 
It will allow us to replace the discontinuous symbol $\phi_\Om = \id_\Om \phi$ 
by its smooth dual symbol.  
In the next proposition we calculate this dual symbol for the case of the quarter-plane 
$\Om = \Om_{\pi/2}$. 
Below we denote by $\z\in\plainC\infty(\R)$ an arbitrary function such that $\z(x) = 1$ 
for $|x|\ge 2$ and $\z(x) = 0$ for $|x|<1$.

\begin{thm} \label{thm:ftrans}
Let $\Om = \Om_{\pi/2}$. The symbol $\phi_\Om^*(x, \xi)$ dual to $\phi_\Om(x, \xi)$ 
satisfies $\phi_\Om^*\in \plainC\infty(\R)$ and 
the following bounds hold:
\begin{align*}
|\p_x^m \p_\xi^n \, \phi_\Om^*(x, \xi)|\lesssim \lu x\ru ^{-1-m}\, \lu \xi\ru^{-1-n}. 
\end{align*}
Furthermore,
\begin{align}\label{eq:b0b1}
\phi_\Om^*(x, \xi) = b_0(x, \xi) + b_1(x, \xi),
\end{align}
where
\begin{align}
 b_0(x, \xi) = &\ \frac{1}{4\pi} \phi(0, 0)\,
\frac{\z(x)}{x}\, \frac{\z(\xi)}{\xi}, \label{eq:b0}\\
|\p_x^m \p_\xi^n \, b_1(x, \xi)|\lesssim &\ \lu x\ru ^{-1-m}\, \lu \xi\ru^{-2-n}
+ \lu x\ru ^{-2-m}\, \lu \xi\ru^{-1-n}.\label{eq:b1diff}
\end{align} 
\end{thm}

\begin{proof} For brevity denote $b = \phi_\Om^*$. 
By definition, 
\begin{align*}
b(x, \xi) = \frac{1}{\pi}\int_0^\infty\int_0^\infty e^{-2i \xi y}   
e^{2i x\eta}  \phi(y, \eta) \, dy d\eta,
\end{align*}
and hence
\begin{align*}
\p_x^m \p_\xi^n \, b(x, \xi) = \frac{1}{\pi}\int_0^\infty\int_0^\infty 
e^{-2i \xi y}   e^{2i x\eta}  (2i\eta)^m (-2iy)^n   \phi(y, \eta) \, dy d\eta.
\end{align*}
We begin with examining the case where $|x|\ge 1$ and $|\xi|\ge 1$.  

Performing several tedious (but elementary) integrations by parts we obtain that 
\begin{align*}
\p_x^m \p_\xi^n \, b(x, \xi)   
= &\ \frac{n! m!}{4\pi} \frac{(-1)^{n+m}}{\xi^{n+1}x^{m+1}}\, \phi(0, 0) \\
&\quad  -\frac{ i\, n!}{2\pi} \frac{(-1)^n}{ \xi^{n+1}} \bigg(\frac{i}{2x} \bigg)^{m+1}
\int_0^\infty    
e^{2i x\eta} \, \p_{\eta}^{m+1}\big((2i\eta)^m \phi(0, \eta)\big) \, d\eta\\
&\ + \frac{i m!}{2\pi} \frac{(-1)^m}{x^{m+1}} 
\bigg(-\frac{i}{2\xi}\bigg)^{n+1}   \int_0^\infty \, e^{-2i\xi y}  \p_y^{n+1} 
 \big((-2iy)^n \phi(y, 0)\big)\, dy 
\\
  + \frac{1}{\pi} 
\bigg(\frac{i}{2x}\bigg)^{m+1} &
\bigg(-\frac{i}{2\xi}\bigg)^{n+1} 
\int_0^\infty\int_0^\infty e^{2i x\eta} 
e^{-2i \xi y}   
\p_y^{n+1} \p_\eta^{m+1} 
 \big((2i\eta)^m  (-2iy)^n \phi(y, \eta)\big) \, dy d\eta.
 \end{align*}
Without writing it out explicitly, observe that another integration by parts ensures that 
each of the last three integrals exhibits an extra decay in $x$ or $\xi$: 
precisely, the first integral 
is bounded by $|x|^{-1}$, the second and the third one -- by $|\xi|^{-1}$. Therefore, 
\begin{align}\label{eq:step1}
b(x, \xi) = \frac{\phi(0,0)}{4\pi x\xi} + \tilde b_1(x, \xi),
\end{align}
 where 
 \begin{align*}
| \p_x^m \p_\xi^n \tilde b_1(x, \xi)|
\lesssim |x|^{-m-1} |\xi|^{-n-2} + |x|^{-m-2} |\xi|^{-n-1},
\quad |x|\ge 1, |\xi|\ge 1.
 \end{align*}
 Since $1-\z(x)$ is compactly supported, 
 the functions $x^{-1}$ and $\xi^{-1}$ in \eqref{eq:step1} can be replaced 
with $\z(x) x^{-1}$ and $\z(\xi)\xi^{-1}$ respectively and the difference can be  
incorporated in $\tilde b_1(x, \xi)$ to form a new symbol that we call $b_1(x, \xi)$. 
This gives the representation \eqref{eq:b0b1}. 

In the case where $|x|<1$ and  $|\xi| >1$ 
the derivatives of $b$ are estimated by integrating by parts in the variable $\xi$ 
only. Similarly, for the case $|x|>1$ and  $|\xi| <1$. If $|x|<1$ and $|\xi|<1$, 
then the derivatives are trivially bounded. Thus the proof of \eqref{eq:b0b1} is complete.
\end{proof}

\section{Compact operators}\label{sect:compact}

We need some information about classes of compact operators on a Hilbert space. 
For a compact self-adjoint operator $A$ we denote by $\l^{(+)}_k(A)$ and 
$-\l^{(-)}_k(A)$, $k = 1, 2, \dots$,  its positive and negative eigenvalues respectively, 
enumerated in non-increasing order. For an arbitrary compact $A$ we define 
its singular values (s-values) as $s_k(A) = \sqrt{\l^{(+)}_k(A^*A)} =  \sqrt{\l^{(+)}_k(AA^*)}$.  
Clearly, for any unitary $U$ the operators $A$ and $AU$ have the same singular values.
Recall the standard notation for the counting functions:
\begin{align*}
n(s; A) = \#\{k=1, 2, \dots: s_k(A)>s\},\quad n_\pm(\l; A) = 
\#\{k=1, 2, \dots: \l^{(\pm)}_k(A)>\l\}, 
\end{align*}
where $ s>0, \l>0$. 
For a pair $A_1, A_2$ of compact operators, the introduced quantities satisfy the following inequalities for all $k, m = 1, 2, \dots$, and all $s_1, s_2>0$, known as the 
Ky Fan inequalities:
\begin{align}
s_{k+m-1}(A_1+A_2)\le &\ s_k(A)+s_m(A),\label{eq:qifan}\\ 
n(s_1 + s_2; A_1+A_2) \le &\ n(s_1; A_1) + n(s_2; A_2).\label{eq:countqifan}
\end{align}
If $A_1, A_2$ are self-adjoint, then
\begin{align}
\l^{(\pm)}_{k+m-1}(A_1+A_2)\le &\ \l^{(\pm)}_k(A)+\l^{(\pm)}_m(A),\notag\\ 
n_\pm(s_1 + s_2; A_1+A_2) \le &\ n_{\pm}(s_1; A_1) + n_\pm(s_2; A_2).\label{eq:saqifan}
\end{align}
The notation $\BS_p, 0<p<\infty $, is standard for the Schatten-von Neumann class of 
compact operators $A$ on a Hilbert space, for whose singular values $s_k(A), k = 1, 2, \dots$, 
the  functional
\begin{align*}
\|A\|_{\BS_p}: = \bigg[\sum_{k=1}^\infty\, s_k^p(A) \bigg]^{\frac{1}{p}}
\end{align*}
is finite. In particular, $\BS_2$ is the usual Hilbert-Schmidt class, that has been mentioned previously. Observe that 
\begin{align}\label{eq:ind}
s_k(A) \le k^{-1/p} \, \|A\|_{\BS_p}.
\end{align}
The class of operators that satisfy the individual bounds $s_k(T) = O(k^{-1/p})$ 
with some $p>0$, is denoted by $\BS_{p, \infty}\supset \BS_p$. The functional 
\begin{align*}
\|A\|_{\BS_{p, \infty}} = \sup_{k}\, s_k(A) k^{\frac{1}{p}}
\end{align*} 
defines a quasi-norm on $\BS_{p, \infty}$. We refer for details on these classes 
to \cite[Chapter 11.6]{BS} and \cite[Ch. 4]{Simon2005}. 

In this paper, however, we work with more general classes of compact operators that 
were studied in \cite{Weidl_1993}. 
Let, for $t\ge 0$,  
\begin{align*}
f(t) = f_\s(t) = \frac{t}{\log^\s (t+e^{\s+1})},\ \quad \s \ge 0.
\end{align*} 
The function $f_\s$ is non-negative, and an elementary calculation shows that $f_\s''(t) <0$ 
for all $t >0$, so that it is (strictly) concave. 
Therefore, it is subadditive:
\begin{align*}
f_\s(t_1+t_2)\le f_\s(t_1) + f_\s(t_2),\quad t_1, t_2\ge 0.
\end{align*}
Define $\GS_{p, \s}, p >0, \s \ge 0$, as the class of compact operators $A$ 
satisfying the bound $s_k (A) = O(f(k)^{-1/p})$. Let 
\begin{align}\label{eq:metric}
\|A\|_{\GS_{p, \s}} = \sup_{k\in\mathbb N}\, s_k(A) \, f(k)^{\frac{1}{p}} = 
\big(\sup_{s>0}\, s^p \, 
f\big( n(s; A)\big)\big)^\frac{1}{p}.
\end{align}
Note that in \cite{Weidl_1993} more general concave functions $f$ are allowed. 
Along with \eqref{eq:metric} introduce the asymptotic quantities
\begin{align}\label{eq:limsupinf}
\begin{cases}
\SfG_{p, \s}(A) = 
\limsup\limits_{k\to\infty} s_k(A) \big(f_\s(k)\big)^{\frac{1}{p}} = 
\limsup\limits_{s\to 0}\,s\, \big(f_\s\big(n(s; A) \big)\big)^{\frac{1}{p}},\\[0.3cm]
\sg_{p,\s}(A) =  
\liminf\limits_{k\to\infty}s_k(A) \big(f_\s(k)\big)^{\frac{1}{p}} = 
\liminf\limits_{s\to 0}\,s\, \big(f_\s\big(n(s; A) \big)\big)^{\frac{1}{p}}.
\end{cases}
\end{align}
If $n(s; A)$ is replaced by $n_{\pm}(\l; A)$, then the notation for these functionals 
is $\SfG^{(\pm)}_{p, \s}$ and  $\sg^{(\pm)}_{p, \s}$. 
Note that the values of these asymptotic functionals 
do not change if the function $f_\s$ is replaced 
by $t\log^{-\s}(t+a)$ with an arbitrary $a\ge 0$.

The next proposition was proved in \cite{Weidl_1993}, but we provide the proof 
for the sake of completeness.

\begin{prop}\label{prop:triangle}
The functional \eqref{eq:metric} defines a metric on $\GS_{p, \s}$. More prescisely, 
for any $A_1, A_2\in\GS_{p, \s}$ the inequality holds:
\begin{align}\label{eq:tri}
\|A_1 + A_2\|_{\GS_{p, \s}}^{\frac{p}{p+1}}\le 
\|A_1\|_{\GS_{p, \s}}^{\frac{p}{p+1}} + \|A_2\|_{\GS_{p, \s}}^{\frac{p}{p+1}}.
\end{align}
Furthermore,
\begin{align}\label{eq:supinf}
\begin{cases}
\big(\SfG_{p, \s}(A_1+A_2)\big)^{\frac{p}{p+1}}\le  
\big(\SfG_{p, \s}(A_1)\big)^{\frac{p}{p+1}} 
+ \big(\SfG_{p, \s}(A_2)\big)^{\frac{p}{p+1}},\\[0.3cm]
\big(\sg_{p, \s}(A_1+A_2)\big)^{\frac{p}{p+1}}\le  \big(\sg_{p, \s}(A_1)\big)^{\frac{p}{p+1}} 
+  \big(\SfG_{p, \s}(A_2)\big)^{\frac{p}{p+1}},
\end{cases}
\end{align}
and the same inequalities hold for $\SfG^{(\pm)}_{p, \s}$ and $\sg^{(\pm)}_{p, \s}$.
\end{prop}

\begin{proof}
Due to \eqref{eq:countqifan}, 
by subadditivity and monotonicity of 
$f$ we obtain for any $s>0$,  $\varepsilon\in (0, 1)$ that 
\begin{align*}
f\big(n(s; A_1+A_2)\big)
\le f\big(n(\varepsilon s; A_1)\big) + f\big(n((1-\varepsilon) s; A_2)\big),
\end{align*} 
and hence
\begin{align*}
s^p\, f\big(n(s; A_1+A_2)\big)
\le \frac{(\varepsilon s)^p}{\varepsilon^p}
\, f\big(n(\varepsilon s; A_1)\big) 
+ \frac{((1-\varepsilon) s)^p}{(1-\varepsilon)^p}
\, f\big(n((1-\varepsilon) s; A_2)\big),
\end{align*}
Taking the supremum over $s>0$, we get 
\begin{align*}
\|A_1+A_2\|_{\GS_{p, \s}}^p\le  \frac{1}{\varepsilon^p}
\, \|A_1\|_{\GS_{p, \s}}^p 
+ \frac{1}{(1-\varepsilon)^p}\, \|A_2\|_{\GS_{p, \s}}^p,
\end{align*}
Minimizing in $\varepsilon$ we obtain \eqref{eq:tri}. 
In the same way, taking $\limsup$ or $\liminf$  we arrive at \eqref{eq:supinf}. 

The inequalities for $\SfG^{(\pm)}_{p, \s}$ and $\sg^{(\pm)}_{p, \s}$ 
are derived from \eqref{eq:saqifan} 
in the same way. 
\end{proof}

The following corollary will be of central importance for us. 

\begin{cor}\label{cor:pert}
If $\SfG_{p, \s}(A_1-A_2) = 0$, then $\SfG_{p, \s}(A_1) = \SfG_{p, \s}(A_2)$,  
$\sg_{p, \s}(A_1) = \sg_{p, \s}(A_2)$, 
and the same equalities hold for $\SfG^{(\pm)}_{p, \s}$ and $\sg^{(\pm)}_{p, \s}$.
\end{cor}

This is an immediate consequence of the bounds \eqref{eq:supinf}.
  
\section{Bounds for pseudo-differential operators}\label{sect:bounds}

In this section we derive phase space-type bounds for the singular values of 
the operator \eqref{eq:opt}.

\subsection{Estimates in terms of lattice norms} 
We rely on the estimates for the operator \eqref{eq:opt} 
in Schatten-von Neumann classes that were  
obtained in \cite{Sobolev2014}.  
Let $\CC_z\subset \R^m$ be a unit cube centred at $z\in\R^m$.  
For a function $h\in\plainL{r}_{\textup{\tiny loc}}(\R^m)$, $r \in (0, \infty),$ 
denote
\begin{equation}\label{eq:brackh}
\begin{cases}
\1 h\1_{r,\delta} =
\biggr[\sum_{n\in\Z^m}
\biggl(\int_{\CC_n} |h(x)|^r d x\biggr)^{\frac{\d}{r}}\biggr]^{\frac{1}{\d}},\ \
0<\d<\infty,\\[0.3cm]
\1 h\1_{r, \infty} =
\sup_{z\in\R^m}
\biggl(\int_{\CC_z} |h(x)|^r dx\biggr)^{\frac{1}{r}},\ \
\d = \infty.
\end{cases}
\end{equation}
These functionals are sometimes called \textit{lattice  
quasi-norms} (norms for $r, \d\ge 1$).  
If $\1 h\1_{r, \d}<\infty$ we say that  
$h\in \plainl{\d}(\plainL{r})(\R^m)$. 
For a symbol $a = a(x, \xi)$ denote
\begin{equation}\label{eq:F}
F_{n}(x, \xi; a)
= \sum_{s, k =0}^n 
|\nabla_{x}^k \nabla_{\xi}^s a(x, \xi)|,\quad n = 1, 2, \dots.
\end{equation}
The next proposition 
is a part of \cite[Theorem 2.6]{Sobolev2014}. 
It is stated in the form convenient for our purposes. 

\begin{prop}\label{prop:symbol} Let $q\in (0, 1]$, and let  
\begin{align}\label{eq:n}
n  = n(q) = [d q^{-1}] + 1. 
\end{align}
If $F_n(a)\in \plainl{q}(\plainL{1})(\R^{2d})$, then   
\begin{equation}\label{eq:symbol}
\|\op^{(t)} (a) \|_{\BS_q} 
\lesssim \1 F_{n}(a)\1_{1, q},
\end{equation}
with a constant independent of $t\in [0, 1]$. 
 \end{prop}

\subsection{Phase space bounds}\label{subsect:phasesp}
Our objective is to obtain bounds in classes $\GS_{p, \s}$ with suitable $p\in (0, 2)$ 
and $\sigma>0$. To this end instead of tiling $\R^{2d}$ with the 
cubes $\CC_\bn, \bn\in \Z^{2d}$, as in Proposition \ref{prop:symbol}, 
we consider a cover of $\Z^{2d}$ by cubes  
$\CC^{(2)}_{\bn} = (-1, 1)^{2d} + \bn$, $\bn\in \Z^{2d}$. 
Recall the notation $\btau = (x, \xi)$ and let 
$\z\in\plainC\infty_0(\CC^{(2)}_{\bold0})$ be a non-negative function such 
that the translations 
$\z_{\bk}(\btau) = \z(\btau-\bk),\, \bk\in \mathbb Z^{2d}$, 
form a partition of unity subordinate to the above cover:  
\begin{align*}
\sum_{\bk\in\Z^{2d}} \z_{\bk}(\btau) = 1,\ \quad \btau\in \R^{2d}.
\end{align*} 
Then the operator $\SA = \SA^{(t)} = \op^{(t)}(a)$ can be represented as  
$\SA = \sum_{\bk} \, \op^{(t)}(\z_\bk a)$. 
Since the parameter $t$ is fixed, we often omit it from the notation. 
For a subset $K\subset \Z^{2d}$, introduce a "trimmed" variant of the operator 
$\SA$:
\begin{align*}
{\sf A}_K = \SA_K^{(t)}:= \op^{(t)}(a_K),\quad a_K(\btau) = 
a(\btau) \sum_{\bk\in K}\, \z_{\bk}(\btau) \, .
%
%
\end{align*} 
In order to estimate the $\BS_q$-quasi-norm of ${\sf A}_K$ introduce the sequence  
\begin{align}\label{eq:vk}
v_{\bk} = \bigg[\int_{\CC^{(2)}_\bk}\, |F_n(\btau; a)|^2 \, d\btau\bigg]^{\frac{1}{2}},
\quad {\sf v} = \{v_\bk\}.
\end{align}
To avoid cumbersome notation, we do not indicate 
the dependence of the sequence $\sv$ on the symbol  $a$ 
and number $n$, but one should always keep this dependence in mind.  

The next fact follows directly from Proposition 
\ref{prop:symbol}.

\begin{prop}\label{prop:vk} Let $q\in (0, 1]$, and let 
the sequence ${\sf v}$ be as defined in \eqref{eq:vk} with the number $n$ 
defined in \eqref{eq:n}. 
 Suppose that ${\sf v}$ belongs to 
$\plainl{q}(K)$. Then 
\begin{align}\label{eq:qtrim}
\| \SA_K^{(t)}\|_{\BS_q}^q 
\lesssim \sum_{\bk\in K}\, v_{\bk}^q,
\end{align}
Moreover, the operator $\SA_K^{(t)}$ and the integral  
$(\SA_K)_{\tiny{\rm int}}(s), s>0$, defined in \eqref{eq:inta}, satisfy the bounds
\begin{align}\label{eq:hstrim}
\| \SA_K^{(t)}\|_{\BS_2}^2\lesssim \sum_{\bk\in K}\, v_{\bk}^2,
\quad 
\| (\SA_K)_{\tiny{\rm int}}(s)\|_{\BS_2}^2
\lesssim s^2 \sum_{\bk\in K}\, v_{\bk}^2.
\end{align}
The constants in the above bounds are independent of the symbol $a$, and of 
$t\in [0, 1]$ and $s>0$.
\end{prop}
 
\begin{proof}
Apply Proposition \ref{prop:symbol} to the operator $A_K$. 
For each $\bm\in \Z^{2d}$ we have 
\begin{align*}
\|F(a_K)\|_{\plainL2(\CC_\bm)}^q\lesssim  \sum_{\bk\in K: |\bk-\bm|_{\infty}\le 1}\, 
\|F(a)\|_{\plainL2(\CC^{(2)}_\bk)}^q =   \sum_{\bk\in K: |\bk-\bm|_{\infty}\le 1}\, v_\bk^q,
\end{align*}
where $|\,\cdot\,|_\infty$ denotes the standard $\plainl\infty$-norm on $\R^{2d}$.
Therefore, 
\begin{align*}
\sum_{\bm\in\Z^{2d}}\|F(a_K)\|_{\plainL1(\CC_\bm)}^q \lesssim 
\sum_{\bk\in K}   \sum_{\bm: |\bk-\bm|_{\infty}\le 1}\, v_\bk^q
\lesssim
\sum_{\bk\in K}\, v_\bk^q.
\end{align*}
Now \eqref{eq:qtrim} follows from \eqref{eq:symbol}. 

The bounds \eqref{eq:hstrim} are consequences of \eqref{eq:hs} and \eqref{eq:int}.
\end{proof}
  
Now we can estimate the singular values of $\SA^{(t)}$ in terms of the classes $\GS_{p, \s}$. 
To this end assume that for the sequence $\sv$ defined in \eqref{eq:vk}, 
the following is satisfied: 
\begin{align}\label{eq:sup}
M_{p,\s}[{\sf v}]:= \sup_{E>0}\, E 
\log^{-\s}\big(E^{-1} + 2\big)\,
\big(\# \{\bk\in 
\Z^{2d}: v_{\bk}>E\}\big)^{\frac{1}{p}}<\infty,
\end{align}
with some $p\in (0, 2)$ and $\s\ge 0$. For such a sequence $\sv$ define 
\begin{align}\label{eq:limsup}
M_{p,\s}^{\circ}[{\sf v}]:=\limsup_{E\downarrow 0}\, \,E 
\log^{-\s}\big(E^{-1} + 2\big)\,
\big(\# \{\bk\in 
\Z^{2d}: v_{\bk}>E\}\big)^{\frac{1}{p}}, 
\end{align}
which is automatically finite.  Observe that in contrast to $M_{p, \s}$, the functional 
$M_{p, \s}^\circ$ is (positively) homogeneous in $\sv$ for all $p>0, \s \ge 0$:  
$M_{p,\s}^{\circ}[\mu {\sf v}] = \mu M_{p,\s}^{\circ}[{\sf v}]$ 
for any $\mu >0$. This is why in what follows we work mostly with $M_{p,\s}^{\circ}$.
 
The proof of the next theorem is inspired by the proof of 
\cite[Theorem 4.6]{Simon2005}.

\begin{thm}\label{thm:limsup} 
Let $p\in (0, 2)$, and let the sequence ${\sf v} = \{v_{\bk}\}$ be defined by \eqref{eq:vk} 
with $n=n(q)$ where $q \in (0, 1], q < p$. 
Suppose that for some $\s\ge 0$, the functional $M_{p, \s}(\sv)$ is finite.  
Then $\op^{(t)}(a)\in \GS_{p, \s}$ and 
\begin{align}\label{eq:qc}
\SfG_{p, \s}(\op^{(t)}(a))\lesssim M_{p,\s}^{\circ}[{\sf v}].
\end{align}
In particular, if $\s = 0$, then $\op^{(t)}(a)\in \BS_{p, \infty}$.
%
%
The constant in the above inequality is independent of the symbol $a$ and the number $t$.  
 \end{thm}

\begin{proof} 
Note that $M_{p,\s}^{\circ}[\mu {\sf v}] = \mu M_{p,\s}^{\circ}[{\sf v}]$ 
for any $\mu >0$.Thus, if $M_{p, \s}^{\circ}[{\sf v}]\not = 0$, then without loss of generality 
we may assume that $M_{p, \s}^{\circ}[{\sf v}] = 1$. As a result, 
throughout the proof we assume that either 
$M_{p,\s}^{\circ}[ {\sf v}]  = 0$ or $M_{p,\s}^{\circ}[{\sf v}] = 1$. 
Denote 
\begin{align}\label{eq:mn}
S_{n} = S_n[{\sf v}] :=\sup_{0<E\le 2^n}\, E 
\log^{-\s}\big(E^{-1} + 2\big)\,
\big(\# \{\bk\in 
\Z^{2d}: v_{\bk}>E\}\big)^{\frac{1}{p}}.
\end{align}
and set $S = S[{\sf v}] : = \sup_{n\in\Z} S_n<\infty$. 
Thus $S_n\downarrow 0$ or $S_n\downarrow 1$ as $n\to -\infty$.

Fix an integer $m$ satisfying the condition 
\begin{align}\label{eq:upperm}
m\le -1-S[{\sf v}],
\end{align}
and split $\SA :=\op^{(t)}(a)$ into two operators: 
\begin{align*}
\SA = \sum_{\bk: v_\bk > 2^{2m}} \op(a\, \z_\bk) 
+ \sum_{\bk: v_\bk \le  2^{2m}} \op(a\, \z_\bk)= : {\sf A}^{(1)}_m + {\sf A}^{(2)}_m. 
\end{align*}
In order to estimate the first term we use Proposition 
\ref{prop:vk} with $K = \{\bk: v_\bk>2^{2m}\}$. The set $K$ is finite, and 
hence the restriction of the sequence $\{\sf v\}$ to $K$ belongs to $\plainl{q}(K)$. 
Therefore, by Proposition \ref{prop:vk}, we have the estimate 
\begin{align*}
\|{\sf A}^{(1)}_m\|_{\BS_q}^q\lesssim \sum_{\bk: v_\bk > 2^{2m}}\, v_\bk^q 
= \sum_{l = 2m}^\infty \, \sum_{\bk: 2^{l}<v_\bk\le 2^{l+1}} v_{\bk}^q.
\end{align*} 
Due to the definition \eqref{eq:mn}, 
\begin{align*}
\#\{\bk: v_\bk > 2^l\}\le  2^{-p l} \log^{p\s}(2^{-p l} +2)\, S_l^p
\lesssim (l_-)^{p\s} 2^{-pl}\, S_l^p,
\end{align*}
where $l_- = \max\{1, -l\}$, 
and hence,
\begin{align}\label{eq:layer}
\sum_{\bk: 2^{l}<v_\bk\le 2^{l+1}}\, v_\bk^q 
\lesssim S_l^p\, 2^{(l+1)q} ( l_-)^{p\s} 2^{-p l}.
\end{align}
Since $p >q$ we obtain the bounds
\begin{align*}
\sum_{l = 2m}^\infty \, \sum_{\bk: 2^{l}<v_\bk\le 2^{l+1}} v_{\bk}^q
\lesssim &\    \sum_{l = 2m}^\infty \, 
S_l^p  2^{(l+1)q} ( l_-)^{p\s} 2^{-p l}\\
\le  &\  
S_m^p |2m|^{p\s}\sum_{l = 2m}^m \, 2^{(q - p)l} 
+  S^p |m|^{p\s}\sum_{l = m}^\infty \,  2^{(q - p)l} 
\\
\lesssim &\  S_m^p |m|^{p\s} 2^{2(q-p)m} + S^p |m|^{p\s} 2^{(q-p)m}\\
= &\ |m|^{p\s} 2^{2(q-p)m} \big(  S_m^p + S^p 2^{-(q-p)m}\big).
\end{align*}
Because of the condition 
\eqref{eq:upperm}, the term in the brackets does not exceed 
$(S_m+2^{-\varepsilon |m|})^p$ 
with some $\varepsilon>0$, which leads to 
the estimate 
\begin{align}\label{eq:a1}
\|{\sf A}^{(1)}_m\|_{\BS_q}^q\lesssim \, |m|^{p\s}\, 2^{2(q-p)m}\,
T_m^p,\quad T_m := S_m +  2^{-\varepsilon  |m|}. 
\end{align}
For the operator ${\sf A}^{(2)}_m$ we use 
\eqref{eq:hstrim} with $K = \{ \bk: v_\bk \le 2^{2m}\}$ to conclude that 
\begin{align*}
\| {\sf A}^{(2)}_m\|_{\BS_2}^2
\le  
\sum_{\bk: v_\bk \le 2^{2m}}\, v_\bk^2
=  \sum_{l \le 2m-1} \, \sum_{\bk: 2^{l}<v_\bk\le 2^{l+1}} v_{\bk}^2.
\end{align*} 
Therefore, from \eqref{eq:layer} with $q=2$ we obtain that 
\begin{align*}
\| {\sf A}^{(2)}_m\|_{\BS_2}^2
 \lesssim   \sum_{l\le 2m-1} \, S_l^p \, |l|^{p\s}\,2^{2(l+1)} 2^{-p l} 
 \lesssim S_{m}^{p} |m|^{p\s}\, 2^{2(2-p)m}, 
\end{align*}
where we have used that $p < 2$. Using \eqref{eq:qifan} and \eqref{eq:ind}, 
we arrive at the bounds 
\begin{align*}
s_{2k}(\SA)\le &\ s_{2k-1}(\SA)\le 
s_k(\SA^{(1)}_m) + s_k(\SA^{(2)}_m)\\ 
\lesssim &\  k^{-\frac{1}{q}} |m|^{\frac{p\s}{q}}\, 2^{2(1-\frac{p}{q})m} 
\, T_m^{\frac{p}{q}} 
+  k^{-\frac{1}{2}} |m|^{\frac{p\s}{2}}\, 
2^{(2- p)m}\, S_m^{\frac{p}{2}}\\ 
\lesssim &\ k^{-\frac{1}{p}} |m|^{\s}
\bigg[
k^{\frac{1}{p}(1-\frac{p}{q})} 
|m|^{-\s(1-\frac{p}{q})}
2^{2(1-\frac{p}{q})m} \, T_m^{\frac{p}{q}} 
+ k^{\frac{1}{p}(1-\frac{p}{2})} |m|^{-\s(1-\frac{p}{2})} 2^{(2-p)m}
S_m^{\frac{p}{2}} 
\bigg]\\
= &\  k^{-\frac{1}{p}} |m|^{\s}
\bigg[
\big(k^{\frac{1}{p}}  
|m|^{-\s} 2^{2m}\big)^{1-\frac{p}{q}}\, T_m^{\frac{p}{q}}
+ \big(k^{\frac{1}{p}}  
|m|^{-\s} 2^{2m} \big)^{1-\frac{p}{2}}\, S_m^{\frac{p}{2}}
\bigg].
\end{align*}
For each $k\ge 1$ take $m$ such that $10^{-1}\le  k^{\frac{1}{p}} |m|^{-\s} 2^{2m}\le 10$. This implies that $m\to-\infty$ as $k\to\infty$, so that the condition 
\eqref{eq:upperm} is satisfied for large $k$.  
Therefore 
\begin{align*}
s_{2k}(\SA)\le s_{2k-1}(\SA)\lesssim k^{-\frac{1}{p}} |m|^{\s}\,
\big( T_m^{\frac{p}{q}}
+ S_m^{\frac{p}{2}}
\big).
\end{align*}
Since $|m|\lesssim \log (k+2)$, this implies the estimate 
\begin{align*}
s_{k}(A) k^{\frac{1}{p}}\frac{1}{\log^\s (k+2)}
\lesssim  T_m^{\frac{p}{q}}
+ S_m^{\frac{p}{2}}.
\end{align*}
Taking $\limsup$ as $k\to\infty$ and remembering that 
$\lim T_m = \lim S_m$ is either $= 1$ or $0$
as $m\to - \infty$, we get the required result. 
\end{proof}

\begin{rem}\label{rem:qc}
We stress the quasiclassical nature of the bound \eqref{eq:qc}. 
Indeed, this bound implies 
that the counting function $n(s; \SA^{(t)})$ is controlled by the volume of the 
``classicaly allowed" region where 
$v_k > s$.  
\end{rem}

\section{Examples of bounds for singular values }
\label{sect:examples}

In this section we illustrate the utility of Theorem \ref{thm:limsup} by two examples. 
\subsection{A variant of Theorem \ref{thm:limsup}}
First we prove a variant of Theorem \ref{thm:limsup} 
that could be more convenient for applications. Let $\rho = \rho(\btau), \btau\in\R^{2d}$,  
be some positive function.  
For $p >0$ and $\s\ge 0$ define 
\begin{align}\label{eq:nocirc}
N_{p, \s}[\rho]:=\sup_{E>0} E 
\log^{-\s}(E^{-1}+2) |\{\btau: \rho(\btau)>E\}|^{\frac{1}{p}}. 
\end{align}
This quantity is finite if and only if the following functional is finite:
\begin{align}\label{eq:ncirc}
N_{p, \s}^\circ[\rho]:=\limsup_{E\downarrow 0}E 
\log^{-\s}(E^{-1}+2) |\{\btau: \rho(\btau)>E\}|^{\frac{1}{p}}. 
\end{align}
Similarly to $M_{p, s}^{\circ}[\sv]$ (see \eqref{eq:limsup}), it 
is homogeneous in $\rho$: $N_{p, \s}^\circ[\mu\rho] 
= \mu N_{p, \s}^\circ[\rho] $ for any $\mu >0$.

\begin{thm}\label{thm:qcl}
Let $\rho = \rho(\btau), \btau\in \R^{2d}$, be a positive function such that 
\begin{align}\label{eq:slow}
\rho(\btau)\asymp \rho(\bnu),\quad \textup{if}\quad |\btau-\bnu|\le 1.
\end{align}
Suppose also that $N_{p, \s}[\rho]<\infty$ for some $p\in (0, 2)$ and some $\s\ge0$. 
Assume that  
\begin{align*}
|\nabla_{\btau}^m a(\btau)|\lesssim \rho(\btau)
\end{align*}
for all $m = 0, 1, \dots$, with a constant potentially depending on $m$. Then 
$\SA^{(t)} = \op^{(t)}(a)\in \GS_{p, \s}$ 
and $\SfG_{p, \s}(\SA^{(t)})\lesssim N_{p, \s}^\circ[\rho]$.
\end{thm}

\begin{proof} 
Let us pick a number $q\in (0, 1]$, $q <p$, and define the sequence $v_\bk, \bk\in\Z^{2d}$, 
as in \eqref{eq:vk}. Clearly, $v_\bk\lesssim \rho(\bk)$. Due to the condition \eqref{eq:slow}, 
we have  
$M_{p, \s}^\circ[\sv]\lesssim N_{p, \s}^\circ[\rho]$. Now the claim follows from Theorem 
\ref{thm:limsup}.
\end{proof}

We apply the above theorem to symbols with a power-like decay.

\subsection{Example 1} 
\begin{cor}\label{cor:homog}
Suppose that for all $m = 0, 1, 2, \dots,$ the symbol $a$ satisfies 
\begin{align}\label{eq:homog}
|\nabla_{\btau}^m a(\btau)|\lesssim \lu \btau \ru^{-\g},
\end{align}
with some $\g > d$, and let $p = 2d\g^{-1}$. Then $N_{p, 0}[\lu \btau \ru^{-\g}]<\infty$ 
and $\op^{(t)}(a)\in\BS_{p, \infty}$ for all $t\in [0, 1]$.
\end{cor}

\begin{proof} We use Theorem \ref{thm:qcl} and estimate the phase space volume 
\begin{align*}
V(E) = \iint \id_{\{\lu \btau\ru^{-\g} > E\}}  \,d\btau. 
\end{align*}
As $V(E) = 0$ 
for $E\ge 1$ we may assume that $E < 1$. 
A straightforward calculation gives the bound 
\begin{align*}
V(E)\le \iint\limits_{|\btau|\le E^{-\frac{1}{\g}}} \, d\btau\lesssim E^{-\frac{2d}{\g}},
\end{align*}
so that $N_{p, 0}[\lu\btau\ru^{-\g}] < \infty$ and hence, by Theorem \ref{thm:qcl}, 
$\op(a):=\op^{(t)}(a)\in \BS_{p, \infty}$, as claimed.
 
\end{proof}

As mentioned in the Introduction, for arbitrary 
domains $\Om\subset \R^2$ with piece-wise smooth boundary, 
the eigenvalues of $\SA_\Om$ decrease in accordance with the bound \eqref{eq:34}. 
Corollary \ref{cor:homog} allows us to establish an analogous bound for domains 
$\Om\in \R^{2d}$ with arbitrary $d\ge 1$. We 
have however to assume that the boundary of $\Om$ has a strictly positive curvature.

\begin{prop}\label{prop:curved} 
Let $\Om\subset\R^{2d}, d\ge 1,$ be a bounded open convex set 
with $\plainC\infty$-boundary of strictly positive 
curvature, and let $\SA_\Om = \op^{\rm w}(\id_\Om)$. 
Then 
$\SA_\Om\in \BS_{p, \infty}$, where $p = 4d(2d+1)^{-1}$, i.e. 
\begin{align*}
s_k(\SA_\Om)\lesssim k^{-\frac{1}{2} - \frac{1}{4d}},\quad k = 1, 2, \dots.
\end{align*}
\end{prop}

\begin{proof} Thanks to Theorem \ref{thm:red}, we just need to estimate singular values of 
the operator $\op^{\rm w}(\id_\Om^*)$ with the dual symbol $ \id_\Om^*$ 
defined in \eqref{eq:dual}.  
To benefit from Corollary \ref{cor:homog}, we need estimates for $ \id_\Om^*$ and its derivatives. 
Using \cite[Theorem 7.7.14]{Hoermander1993} and \cite[Corollary 7.7.15]{Hoermander1993} 
for the Fourier transform of $\id_\Om$ we can write the bounds 
\begin{align*}
|\nabla_{\btau}^m \,  \id_\Om^*(\btau)|\lesssim \lu \btau\ru^{-\frac{2d+1}{2}},\quad 
m = 0, 1, \dots.
\end{align*}
Now it follows from 
Corollary \ref{cor:homog} that $\SA_\Om\in\BS_{p, \infty}$ with the number $p$ specified above, 
as claimed.
\end{proof}

For $d = 1$ the obtained bound agrees with the estimate \eqref{eq:34}. 
We point out that the proofs of \eqref{eq:34} for $d = 1$ in \cite{Ramanathan1993} and 
\cite{Derkach2024} followed different methods and 
did not require strict positivity of the curvature. 

\subsection{Example 2}
The next corollary of Theorem 
\ref{thm:qcl} 
will be directly used in the proof of Theorem  
\ref{thm:main}.

\begin{cor}\label{cor:homo}
Suppose that for all $m, n = 0, 1, 2, \dots,$ the symbol $a$ satisfies
\begin{align}\label{eq:homo}
|\nabla_x^m \nabla_\xi^n a(\btau)|\lesssim \rho(\btau),\quad \rho(x, \xi) = \lu x\ru^{-\a}\, \lu\xi\ru^{-\b}
+ \lu x\ru^{-\b}\, \lu\xi\ru^{-\a} ,
\end{align}
with some $\a, \b> d/2$, and let $p = \max\{d\a^{-1}, d\b^{-1}\}$. Denote $\s = p^{-1}$ 
if $\a=\b$ and $\s = 0$ if $\a\not = \b$. 
Then  $N_{p, \s}[\rho] < \infty$ and $\op^{(t)}(a)\in \GS_{p, \s}$ for 
all $t\in [0, 1]$. 
\end{cor}

\begin{proof} 
As in the proof of Corollary 
\ref{cor:homog}, 
it suffices to estimate the phase space volume  
\begin{align*}
\iint \id_{\{\lu x\ru^{-\a} \lu\xi\ru^{-\b} + 
\lu x\ru^{-\b}\, \lu\xi\ru^{-\a} > E\}}  \,dx d\xi. 
\end{align*}
Since 
\begin{align*}
\id_{\{\lu x\ru^{-\a} \lu\xi\ru^{-\b} + 
\lu x\ru^{-\b}\, \lu\xi\ru^{-\a} > E\}}  
\le \id_{\{\lu x\ru^{-\a} \lu\xi\ru^{-\b} > E/2\}}  
+ \id_{\{ 
\lu x\ru^{-\b}\, \lu\xi\ru^{-\a} > E/2\}},  
\end{align*}
we can estimate the volumes separately for $\lu x\ru^{-\a} \lu\xi\ru^{-\b}$ 
and $\lu x\ru^{-\b} \lu\xi\ru^{-\a}$. Let us  estimate, for example,  
\begin{align*}
V(E) = \iint \id_{\{\lu x\ru^{-\a} \lu\xi\ru^{-\b}> E\}}  \,dx d\xi,
\end{align*}
for $E <1$. 
Suppose for definiteness that 
$\a \le  \b$ and write:  
\begin{align}\label{eq:volest}
V(E)\lesssim &\ \int\limits_{ |x|\le E^{-\frac{1}{\a}}} \hskip 0.2cm
\int\limits_{|\xi|\le E^{-\frac{1}{\b}}\lu x \ru^{-\frac{\a}{\b}}}\, d\xi \, dx
\lesssim E^{-\frac{d}{\b}}\, \int\limits_{ |x|\le E^{-\frac{1}{\a}}} 
 \lu x \ru^{-\frac{ d\a}{\b}}\, dx.
\end{align}
If $\a < \b$, then the right-hand side does not exceed 
\begin{align*}
E^{-\frac{d}{\b}}\, \int\limits_{ |x|\le E^{-\frac{1}{\a}}} 
 \lu x \ru^{-\frac{ d\a}{\b}}\, dx
 \lesssim E^{-\frac{d}{\b}}\, E^{-\frac{1}{\a}\big(-\frac{d\a}{\b} + d\big)}
 = E^{-\frac{d}{\a}}, 
\end{align*}
which means that $N_{p, 0}[\lu x\ru^{-\a} \lu\xi\ru^{-\b}]<\infty$ with $p = d/\a$, see definition 
\eqref{eq:ncirc}. The inclusion 
$\op^{(t)}(a)\in \GS_{p, 0} = \BS_{p, \infty}$ is a consequence of Theorem \ref{thm:qcl}. 
If $\a = \b$, then the right-hand side of \eqref{eq:volest} 
equals 
\begin{align*}
E^{-\frac{d}{\b}}\, \int\limits_{ |x|\le E^{-\frac{1}{\a}}} 
 \lu x \ru^{-d}\, dx
\lesssim  E^{-\frac{d}{\b}}\, \log \big(E^{-1} + 2). 
\end{align*}
By definition \eqref{eq:ncirc} this means that 
$N_{p, \s}[\lu x\ru^{-\a} \lu\xi\ru^{-\b}]<\infty$. 
Similarly, $N_{p, \s}[\lu x\ru^{-\b} \lu\xi\ru^{-\a}]<\infty$, 
and hence, $N_{p, \s}(\rho) <\infty$, as claimed. 
The inclusion 
$\op^{(t)}(a)\in \GS_{p, \s}$ is a consequence of Theorem \ref{thm:qcl} again. 
\end{proof}
   
\section{From the Weyl to Kohn-Nirenberg symbol}\label{sect:kntow}


In this section we consider smooth $\plainL2$-symbols $a$ and demonstrate that,
modulo some mild assumptions, switching  
from the Weyl to Kohn-Nirenberg symbol does not affect the spectral asymptotics, 
 
Under the assumption $a\in \plainL{2}(\R^{2d})$ 
the operator $\op^{\rm w}(a) - \op^{\rm l}(a)$ is Hilbert-Schmidt, see 
Lemma \ref{lem:hs}.  
%
%
Let us estimate its Hilbert-Schmidt norm. Below $\plainW{2}{2}(\R^{2d})$ denotes the standard 
Sobolev space on $\R^{2d}$. 

\begin{lem}\label{lem:hsb}
Suppose that $a\in \plainW{2}{2}(\R^{2d})$. Then the operator 
$\op^{\rm w}(a) - \op^{\rm l}(a)$ is Hilbert-Schmidt and 
\begin{align*}
\|\op^{\rm w}(a) - \op^{\rm l}(a)\|_{\BS_2}^2\le \frac{1}{4(2\pi)^d} \iint\,
\big|((\nabla_x\cdot\nabla_\xi)\, a)(x,  \xi)\big|^2 \, dx d\xi.
\end{align*}
\end{lem}

\begin{proof} 
By Lemma \ref{lem:hs}, it suffices to conduct the proof for $a\in \CS(\R^{2d})$.  
Denote 
\begin{align}\label{eq:ampl}
p(x, y; \xi) = a\bigg(\frac{x+y}{2}, \xi\bigg) - a(x, \xi),
\end{align}
so that $\op^{\rm w}(a) - \op^{\rm l}(a) = \op^{\rm a}(p)$. 
Rewrite the amplitude:
\begin{align*}
p(x, y, \xi) = &\ \frac{1}{2}\int_{0}^1 \, 
(y-x)\cdot\nabla_x a\bigg(x+ s\frac{y-x}{2}, \xi\bigg)\, ds\\
 = &\ \int_{0}^{1/2} \, 
(y-x)\cdot\nabla_x a\big(x+ s(y-x), \xi\big)\, ds. 
\end{align*}
Thus, integrating by parts, we can represent the kernel $\CP(x, y)$ of the operator 
$\op^{\rm a}(p)$ as follows: 
\begin{align*}
\CP(x, y) = &\  - \frac{i}{(2\pi)^{d}}\, \int e^{i(x-y)\cdot\xi}\, \int_0^{1/2} (\nabla_x\cdot\nabla_\xi) 
a\big(x+ s(y-x),  \xi\big)\,ds \, d\xi. 
\end{align*}
This implies that 
\begin{align*}
\op^{\rm a}(p) = \int_0^{1/2} \ \op^{(s)}(w) ds,\quad w(x, \xi) =  
- i (\nabla_x\cdot\nabla_\xi) a(x, \xi).
\end{align*}
The required bound follows from \eqref{eq:int}.    
This completes the proof. 
\end{proof}

Now we can estimate singular values of the difference 
$\op^{\rm w}(a) - \op^{\rm l}(a)$. To this end 
we have to impose stronger conditions on $a$ than in Lemma \ref{lem:hsb}. 

%
%
Let $q\in (0, 1]$, and let 
the sequence ${\sf v} = \{v_{\bk}\}$ be defined by \eqref{eq:vk} 
with $n=n(q)$, see \eqref{eq:n} for definition.  
%
Apart from $v_\bk$ we  also use
\begin{align}\label{eq:w}
w_\bk = \bigg[\int_{\CC^{(2)}_\bk}\,
|(\nabla_x\cdot\nabla_\xi) a(\btau)|^2 \, 
d\btau\bigg]^{\frac{1}{2}},\quad \bk\in\Z^{2d}. 
\end{align}
Recall that $M_{\g, \s}[{\sf v}]$ and $M_{\g, \s}^\circ[{\sf v}]$ are defined in 
\eqref{eq:sup} and \eqref{eq:limsup} 
respectively.

\begin{thm}\label{thm:difference}
Suppose that for some 
$p \in (0, 2)$, some $q\in (0, 1]$, such that $q < p$, and some $\s\ge 0$,  we have  
$M_{p, \s}[\sf v] <\infty$ and $ M_{p, \s}^{\circ}[\sf w] = 0$. 
Then 
\begin{align}\label{eq:difference}
\SfG_{p, \s}\big( \op^{\rm w}(a) - \op^{\rm l}(a)\big) = 0.
\end{align}
\end{thm}

\begin{proof}  
If $M_{p, \s}^{\circ}[{\sf v}] = 0$, then by Theorem \ref{thm:limsup}, the functional 
$\SfG_{p, \s}$ equals zero for each of the operators $\op^{\rm w}(a)$ and $\op^{\rm l}(a)$ 
and hence \eqref{eq:difference} holds by \eqref{eq:supinf}. Thus further on we suppose that 
$M_{p, \s}^{\circ}[{\sf v}] \not = 0$, and therefore, as in the proof of Theorem \ref{thm:limsup}, 
we may assume that $M_{p, \s}^{\circ}[{\sf v}] = 1$. 

To a large extent, our proof follows that of Theorem \ref{thm:limsup}. In particular, we use 
the notation \eqref{eq:mn} and note that $S_n[{\sf v}]\to 1$ and $S_n[{\sf w}]\to 0$ 
as $n\to-\infty$. 
 
Fix an integer $m$ satisfying \eqref{eq:upperm} 
and split $\SP:=\op^{\rm a}(p)$ with the amplitude 
\eqref{eq:ampl} into two operators: 
\begin{align*}
\SP = &\ \SP^{(1)}_m + \SP^{(2)}_m,\\
\SP^{(1)}_m = &\ \sum_{\bk: w_\bk > 2^{2m}} \big(\op^{\rm w}(a_\bk) 
- \op^{\rm l}(a_\bk)\big),\\
\SP^{(1)}_m= &\ 
\sum_{\bk: w_\bk \le  2^{2m}} \big(\op^{\rm w}(a_\bk) 
- \op^{\rm l}(a_\bk)\big),
\end{align*}
where $a_\bk$ are as defined in Sect. \ref{subsect:phasesp}. 
Applying the same argument as in the proof of Theorem \ref{thm:limsup} to each 
of the two operators in $\SP^{(1)}_m$, we obtain the bound
\begin{align*}
\|\SP^{(1)}_m\|_{\BS_q}^q\lesssim \sum_{\bk: w_\bk > 2^{2m}}\, v_\bk^q 
\lesssim \sum_{\bk: v_\bk > 2^{2m}}\, v_\bk^q, 
\end{align*} 
where we have remembered that $w_\bk \le v_\bk$. Thus $\SP^{(1)}_m$ satisfies the bound 
\eqref{eq:a1}: 
\begin{align*}
\|\SP^{(1)}_m\|_{\BS_q}^q\lesssim |m|^{p\s}\, 2^{(q-p)m} \, T_m^p, 
\end{align*}
where $T_m = S_m[{\sf v}] + 2^{-\varepsilon|m|}$ with some $\varepsilon>0$. 
For the operator $\SP^{(2)}_m$ we use Lemma \ref{lem:hsb}:
\begin{align*}
\| \SP^{(2)}_m\|_{\BS_2}^2\lesssim &\ \sum_{\bk: w_\bk\le 2^{2m}}\, w_\bk^2
 = \sum_{l \le 2m-1} \, \sum_{\bk: 2^{l}<w_\bk\le 2^{l+1}} w_{\bk}^2.
\end{align*} 
Repeating the argument in the proof of Theorem \ref{thm:limsup}, we obtain 
the bound
 \begin{align*}
\| \SP^{(2)}_m\|_{\BS_2}^2 \lesssim  2^{2(2-p)m}\, |m|^\s\, S_m^p,\quad 
S_m = S_m[{\sf w}].
\end{align*}
As in the proof of Theorem 
\ref{thm:limsup}, this leads to the estimate
\begin{align*}
s_{2k}(\SP)\le s_{2k-1}(\SP)
\lesssim  k^{-\frac{1}{p}} |m|^{\s}
\bigg[
\big(k^{\frac{1}{p}}  
|m|^{-\s} 2^{2m}\big)^{1-\frac{p}{q}}\, T_m^{\frac{p}{q}}
+ \big(k^{\frac{1}{p}}  
|m|^{-\s} 2^{2m} \big)^{1-\frac{p}{2}}\, S_m^{\frac{p}{2}}
\bigg].
\end{align*}
For each $k\ge 1$ take $m$ such that 
$10^{-1+r}\le  k^{\frac{1}{p}} |m|^{-\s} 2^{2m}\le 10^{1+r}$ with some $r>0$. 
Therefore 
\begin{align*}
s_{2k}(\SP)\le s_{2k-1}(\SP)\lesssim k^{-\frac{1}{p}} |m|^{\s}\,
\big(
 10^{r(1-\frac{p}{q})}T_m^{\frac{p}{q}} 
+ 10^{r(1-\frac{p}{2})} S_m^{\frac{p}{2}}
\big),
\end{align*}
with a constant independent of $r$. 
Since $|m|\lesssim \log (k+2) + r$, this implies the estimate 
\begin{align*}
s_{k}(\SP) \lesssim  k^{-\frac{1}{p}}(\log (k+2) + r)^\s
\big(
 10^{r(1-\frac{p}{q})}T_m^{\frac{p}{q}} 
+ 10^{r(1-\frac{p}{2})} S_m^{\frac{p}{2}}\big).
\end{align*}
Taking $\limsup$ as $k\to\infty$ and remembering that 
$\lim T_m = 1$ and $\lim S_m = 0$ 
as $m\to - \infty$, we obtain that 
\begin{align*}
\limsup_{k\to\infty} 
s_{k}(\SP)\, k^{\frac{1}{p}}\frac{1}{\log^\s (k+2)}
\lesssim  10^{r(1-\frac{p}{q})}.
\end{align*}
Since $r>0$ is arbitrary and $p >q$, the left-hand side equals zero. 
This implies \eqref{eq:difference}.
\end{proof}

\section{Proof of Theorem \ref{thm:main}}\label{sect:proofs}

Employing the notation of Sect. \ref{sect:compact}, we can recast Theorem 
\ref{thm:main} as follows.

\begin{thm}\label{thm:restate} 
Let $\Om = \Om_\t\subset \R^2$ be a sector with the opening angle 
$\t\in (0, \pi)$ or $\t\in (\pi, 2\pi)$, and let  $\phi\in\plainC\infty_0(\R^2)$. 
Then 
\begin{align}\label{eq:mainf}
\SfG_{1,1}^{(\pm)}\big(\SA_{\Om}[\phi]\big) 
= \sg_{1,1}^{(\pm)}\big(\SA_{\Om}[\phi]\big) = \frac{|\phi(0, 0)|}{4\pi^2}.
\end{align}
If $\phi(0, 0) = 0$, then $\SA_{\Om}[\phi]\in \BS_{1, \infty}$. 
\end{thm}

The proof consists of several steps detailed below.

\subsection{Reduction to $\t = \pi/2$}\label{subsect:reduction} 
The first step is to 
 show that it suffices to prove Theorem \ref{thm:restate} for $\t = \pi/2$ only. 
If $\t\in (0, \pi)$, then this conclusion follows from 
Remark \ref{rem:sympl}(\ref{item:sympl}).   
Now suppose that $\t\in (\pi, 2\pi)$. Rewrite:
\begin{align*}
\SA_{\Om_\t}[\phi] = \op^{\rm w}(\phi) - \SA_{\L_\t}[\phi],\quad \textup{with}\quad 
\L_\t = \{(x, \xi): \arg(x+i\xi) \in (\t, 2\pi)\}.
\end{align*}
Since $\phi\in\plainC\infty_0(\R^2)$, it satisfies \eqref{eq:homo} for any $\a>0, \b>0$, 
so $\op^{\rm w}(\phi)\in \BS_{p, \infty}$ for all $p >0$, 
and therefore, $\SfG_{1,1}\big(\op^{\rm w}(\phi)\big) = 0$. 
By Corollary \ref{cor:pert},
\begin{align*}
\SfG_{1,1}^{(\pm)}\big(\SA_{\Om_\t}[\phi]\big)= \SfG_{1,1}^{(\pm)}\big(-\SA_{\L_\t}[\phi]\big)
= \SfG_{1,1}^{(\mp)}\big(\SA_{\L_\t}[\phi]\big),
\end{align*} 
and the same identities hold for the functionals $\sg_{1, 1}^{(\pm)}$. 
The opening angle of the sector $\L_\t$ is $2\pi-\t\in (0, \pi)$, 
and hence Remarks \ref{rem:sympl}\eqref{item:rot} and \ref{rem:sympl}(\ref{item:sympl}) 
ensure that the formula \eqref{eq:mainf} for the sector 
$\Om_{\pi/2}$ implies the same formula for the sector $\L_\t$, and hence for $\Om_\t$.
 
A similar argument, based on Proposition \ref{prop:triangle} 
shows that the inclusion $\SA_\Om[\phi]\in\BS_{1, \infty}$ is equivalent to 
$\SA_{\pi/2}[\phi]\in\BS_{1, \infty}$.

In the rest of this section 
we focus on the proof of Theorem \ref{thm:restate} for the case where 
$\Om = \Om_{\pi/2}$.

\subsection{Proof of $\SA_\Om[\phi]\in\BS_{1, \infty}$ under the assumption 
$\phi(0, 0) = 0$}\label{subsect:straight}
In view of 
Theorem \ref{thm:red}, $s_k(\SA_\Om) = s_k(\op^{\rm w}(\phi_\Om^*)), k= 1, 2, \dots$, 
where $ \phi_\Om^*$ is the dual of $\phi_\Om$. 
Since $\phi(0,0) = 0$, by Theorem \ref{thm:ftrans}, $\phi_\Om^*(\btau) = b_1(\btau)$, 
and hence it satisfies \eqref{eq:homo} 
with $\a=1, \b = 2$. Thus Corollary \ref{cor:homo} implies that 
$\SA_\Om\in\GS_{1, 0} = \BS_{1, \infty}$.
The proof is complete. 
\qed

\subsection{Spectral symmetry} 
We begin the proof of \eqref{eq:mainf} with demonstrating that 
the spectrum of $\SA_\Om$ is asymptotically symmetric about $\l = 0$.
Denote $\id_\pm = \id_{\R_\pm}$. The same symbol will be used for the orthogonal projections 
in $\plainL2(\R)$ onto $\plainL2(\R_\pm)$. 

\begin{lem}\label{lem:diag} 
$\id_\pm \op^{{\rm w}}(\phi_\Om)\id_\pm\in \BS_{1, \infty}$.  
\end{lem}

\begin{proof} For definiteness consider the projection $\id_+$ only.  
Since $\phi$ is compactly supported, we have 
$\phi(x, \xi) = \id_{(-R, R)}(x)\, \phi(x, \xi)$ for some $R>0$, so we can rewrite the symbol of 
$\id_+ \op^{{\rm w}}(\id_\Om\phi) \id_+$ as 
\begin{align*}
\id_+(x)\, \id_{(-R, R)}\bigg(\frac{x+y}{2}\bigg)\, 
\id_+(y)\, \phi_\Om\bigg(\frac{x+y}{2}, \xi\bigg).
\end{align*}
This symbol equals zero outside the square 
\begin{align*}
\{(x, y): 0<x<2R,\, 0<y< 2R\}.
\end{align*} 
Therefore, 
\begin{align*}
\id_+ \op^{{\rm w}}(\phi_\Om) \id_+ = \id_{(0, 2R)}\, 
\op^{{\rm w}}(\phi_\Om) \, \id_{(0, 2R)}.
\end{align*}
Let $\eta\in\plainC\infty_0(\R)$ be a function such that $\eta(t) = 1$ for $t\in [-4R, 4R]$. Then  
$\id_{(0, 2R)}(x)\,\id_{(0, 2R)}(y) = \id_{(0, 2R)}(x)\,\id_{(0, 2R)}(y)\eta(x-y)$. 
This leads to the representation
\begin{align*}
\id_+ \op^{{\rm w}}(\phi_\Om) \id_+ = \id_+ \op^{{\rm a}}(p)\,\id_+,\quad 
p(x, y, \xi) = \phi_\Om\bigg(\frac{x+y}{2}\bigg)\, \eta(x-y). 
\end{align*}
Estimate singular values of $\op^{{\rm a}}(p)$. 
Let the unitary operator $U$ be as in Theorem \ref{thm:red}, i.e. $(Uf)(x) = f(-x)$. 
Since 
the multiplication by a unitary operator does not change the $s$-values, 
we consider $\op^{{\rm a}}(p) U$ instead. 
By Remark \ref{rem:dual}, this operator coincides with $\op^{\rm w}(b)$, where 
\begin{align*}
b(x, \xi) : = \phi_\Om^*(x, \xi) \, \eta(-2x),
\end{align*}
and $\phi_\Om^*$ is the dual symbol of $\phi_\Om$, as defined in \eqref{eq:dual}. 
Since $\eta\in\plainC\infty_0$, Theorem \ref{thm:ftrans} ensures that 
 the symbol $b$ satisfies the bound
\begin{align*}
\p_x^m \p_\xi^n \, b(x, \xi) \lesssim \lu x \ru^{-\b}\, \lu \xi\ru^{-1},\quad 
m, n = 0, 1, \dots,
\end{align*}
with an arbitrary $\b >0$. According to Corollary \ref{cor:homo}, 
$\op^{\rm w}(b)$ belongs to $\BS_{1, \infty}$, and hence so does 
$\id_+ \op^{{\rm w}}(\phi_\Om) \id_+$, 
as claimed.
\end{proof}

The above lemma entails the following asymptotic symmetry. Let $\SfG, \SfG^{\pm}$, and  
$\sg, \sg^{\pm}$ be as defined in Sect. \ref{sect:compact}.

\begin{prop}\label{prop:sva}
Let $\Om = \Om_{\pi/2}$. Then 
\begin{align*}
\SfG_{1, 1}^+(\SA_\Om) = \SfG_{1, 1}^-(\SA_\Om) = \frac{1}{2}\SfG_{1, 1}(\SA_\Om),\quad 
\sg_{1, 1}^+(\SA_\Om) = \sg_{1, 1}^-(\SA_\Om) = \frac{1}{2} \sg_{1, 1}(\SA_\Om). 
\end{align*}

\end{prop} 

\begin{proof} 
Represent $\SA_\Om$ in the block-matrix form:
\begin{gather*}
\SA_\Om =  \SA_\Om^{(1)} + \SA_\Om^{(2)},\\
\SA_\Om^{(1)}
= \begin{pmatrix}
0&\id_+\SA_\Om \id_-\\
\id_-\SA_\Om \id_+& 0
\end{pmatrix},\quad 
\SA_\Om^{(2)}
= \begin{pmatrix}
\id_+\SA_\Om \id_+&0\\
0& \id_-\SA_\Om \id_-
\end{pmatrix}.
\end{gather*}
By Lemma \ref{lem:diag}, $\SA_\Om^{(2)}\in \BS_{1, \infty}$, so that 
$\SfG_{1, 1}(\SA_\Om^{(2)}) = 0$. 

As for the operator $\SA_\Om^{(1)}$, its spectrum is symmetric, i.e. 
$n_+(\l; \SA_\Om^{(1)}) = n_-(\l; \SA_\Om^{(1)})$. 
Indeed, 
\begin{align*}
W^* \SA_\Om^{(1)} W = - \SA_\Om^{(1)},\ \quad \textup{with the unitary}\quad 
W = \begin{pmatrix}
\id_+ &0\\0&-\id_-
\end{pmatrix}. 
\end{align*}
Thus 
$n(s; \SA_\Om^{(1)}) = 2n_\pm(s;, \SA_\Om^{(1)})$ for all $s >0$, and therefore, 
\begin{align*}
\SfG^\pm_{1, 1}(\SA_\Om^{(1)}) = \frac{1}{2} \SfG_{1,1}(\SA_\Om^{(1)}),\ 
\sg^\pm_{1, 1}(\SA_\Om^{(1)}) = \frac{1}{2} \sg_{1,1}(\SA_\Om^{(1)}).
\end{align*} 
Since $\SfG_{1, 1}(\SA_\Om^{(2)}) = 0$, by Corollary \ref{cor:pert}, in the equalities above 
the operator $\SA_\Om^{(1)}$ can be replaced with $\SA_\Om$ which completes the proof. 
\end{proof}

\subsection{The Schr\"odinger operator}
At the last stage of the proof we link the counting function 
for the operator $\SA_\Om$ with that of some Schr\"odinger operator on $\plainL2(\R)$. In the subsection 
we provide the necessary facts about discrete spectrum of the appropriate Schr\"odinger operator.  

Let $V=V(x), x\in\R$, be a bounded  real-valued function such that 
$V(x)\to 0$ as $|x|\to\infty$. 
Let $H$ be the self-adjoint differential operator associated with the closed quadratic form 
\begin{align}\label{eq:hform}
\int_{\R}\, \big(|u'(x)|^2 - V(x) |u(x)|^2\big)  \, dx,\quad u\in \plainW{1}{2}(\R).
\end{align}
On its domain, the operator $H$ is given by the differential expression 
\begin{align*}
H u = - u'' - Vu,
\end{align*}
i.e. it is the Schr\"odinger operator with a potential $-V$. We write $H_0$ if $V = 0$. 
The operator $H$ can have only discrete spectrum below $\l = 0$. 
%
%
Let $\#(\l, V), \l >0,$ be the number of eigenvalues below $-\l <0$. 
Along with the operator $V$ consider the 
compact operator 
\begin{align*}
T(\l) = (H_0+\l)^{-1/2}\, V\, (H_0+\l)^{-1/2},\ \l>0.
\end{align*}
The next identity, known in the literatrure as the Birman-Schwinger principle, 
relates the counting functions for $H$ and for $T(\l)$:
\begin{align}\label{eq:bsch}
\#(\l, V) = n_+(1, T(\l)).
\end{align}
If $V\ge 0$, then the operator $T(\l)$ can be factorized in a natural way:
\begin{align}\label{eq:tl}
T(\l) = \big(\op^{\rm l}(P_0)\big)^*\, \op^{\rm l}(P_0),\quad \textup{where}\quad 
P_0(x, \xi) = \sqrt{V(x)}\, \frac{1}{\sqrt{\xi^2+\l}}.
\end{align}
Thus the identity \eqref{eq:bsch} can be recast as follows
\begin{align}\label{eq:bsch1}
\#(\l, V) = n(1, \op^{\rm l}(P_0)).
\end{align}
We shall use this equality to obtain spectral information on $\op^{\rm l}(P_0)$ from 
spectral properties of the Schr\"odinger operator. To this end we 
define Schr\"odinger operators on 
sub-intervals of the real line. 

For $R>0$ let $H^{(<)}_R$, $H^{(\circ)}_R$, $H^{(>)}_R$ be the self-adjoint operators 
associated with the 
quadratic form \eqref{eq:hform} on the domains $\plainW{1}{2}_0(-\infty, -R)$, 
$\plainW{1}{2}_0(-R, R)$ 
and $\plainW{1}{2}_0(R, \infty)$ respectively. In other words, 
the above operators are Schr\"odinger operators on $(-\infty, -R)$, $(-R, R)$ and $(R, \infty)$ 
with the Dirichlet condition at $x = -R$ and $x = R$. 
Denote by $\#^{(<)}_R(\l, V)$, $\#^{(\circ)}_R(\l, V)$ and $\#^{(>)}_R(\l, V)$ the number of eigenvalues of each operator below $-\l<0$. 
Recall the well-known \textit{decoupling principle}:
\begin{align}\label{eq:decoup}
\big|\#(\l, V) - &\ \big[
\#^{(<)}_R(\l, V) + \#^{(\circ)}_R(\l, V) + \#^{(>)}_R(\l, V)
\big]\big|\le 2.
\end{align}
This inequality is based on standard facts of theory of extensions for symmetric operators,
see e.g. \cite[Ch. 8]{AkhGlaz1993} and \cite[Ch. 9.3]{BS}. 
We need 
the following formula for $\#^{(\circ)}_R(\l; V)$, 
see \cite[Theorem 6.4]{Sob_Sol_2002} and also \cite[Lemma 5.2]{Sobolev1991}:

\begin{prop}
Suppose $V\in \plainC1(\R)$ is a non-negative 
function, and that $0<R <\infty$. Then for any $\l\in\R$ we have 
\begin{align}\label{eq:prufer}
\bigg| \#^{(\circ)}_R(\l, V) - &\ \frac{1}{\pi}\int_{-R}^R\, \sqrt{V(t)}\, dt\bigg|\notag\\
&\ \qquad\qquad\le \int_{-R}^R\, \frac{1}{4\pi}\frac{|V'(t)|}{(V(t)+|\l|)}\, dt 
+ \frac{6 R \sqrt{|\l|+1}}{\pi} + 1.
\end{align}
\end{prop}

We shall apply the formula \eqref{eq:prufer} with a very specific choice of the potential $V$:

\begin{cor}\label{cor:crit}
Let 
\begin{align*}
V(x) = \frac{1}{1+ x^2}. 
\end{align*}
Then for all $g\ge 1$, we have 
\begin{align}\label{eq:crit}
|\#(1, g V) - \frac{\sqrt g}{\pi} \log g |\lesssim  \sqrt g,
\end{align}
with a constant independent of $g$. 
\end{cor}

\begin{proof}
We use the decoupling principle \eqref{eq:decoup} with $R = \sqrt g$. Then $g V(x) < 1$ for 
$|x| >R$, so that 
\begin{align*}
\int_{-\infty}^{-R} \big(|u'(x)|^2 + (1 - gV(x))|u(x)|^2\big) dx >0, 
\quad \textup{for all}\quad u\in \plainW{1}{2}_0(-\infty, -R),
\end{align*}
and the same holds for the interval $(R, \infty)$. Therefore,  
\begin{align}\label{eq:out}
\#^{(<)}_R(1, gV) = \#^{(>)}_R(1, gV) = 0.
\end{align}
To estimate $\#^{(\circ)}_R(1, gV)$ we use \eqref{eq:prufer}. 
Find the integral on the left-hand side of \eqref{eq:prufer}:
\begin{align*}
\int_{-\sqrt g}^{\sqrt g} \frac{1}{\sqrt{1+x^2}}\, dx
= \log g + O(1). 
\end{align*}
The  right-hand side of \eqref{eq:prufer} does not exceed
\begin{align*}
\int_{|x|< \sqrt g}\, \frac{|x|}{1+x^2}\, dx 
+ \sqrt g + 1 
\lesssim \log g + \sqrt g + 1. 
\end{align*}
Together with \eqref{eq:out}, the formulas \eqref{eq:decoup} and \eqref{eq:prufer} entail \eqref{eq:crit}. 
\end{proof}

\subsection{Proof of formula \eqref{eq:mainf}}\label{subsect:proofs} 
Recall that $\t = \pi/2$. Without loss of generality we assume that $\phi(0, 0) = 1$.
By virtue of Proposition \ref{prop:sva}, the equalities \eqref{eq:mainf} will follow if we prove that 
\begin{align}\label{eq:gs}
\SfG_{1,1}(\SA_\Om) = \sg_{1,1}(\SA_\Om) = \frac{1}{2\pi^2}.
\end{align}
To avoid writing too many formulas, we discuss only $\SfG_{1,1}$. 
The functional $\sg_{1,1}$ can be 
considered in the same way. Theorem \ref{thm:ftrans} implies that 
$\SfG_{1,1}(\SA_\Om) = \SfG_{1, 1}(\op^{\rm w}(b))$ for the dual symbol $b=\phi_\Om^*$ defined 
in \eqref{eq:dual}. 
%
%
In the representation \eqref{eq:b0b1} the symbol $b_1$ satisfies 
\eqref{eq:b1diff}, and hence, by virtue of 
Corollary \ref{cor:homo}, $\SfG_{1,1}\big(\op^{\rm w}(b_1)\big) = 0$. Thus it follows from 
Corollary \ref{cor:pert} that 
\begin{align*}
\SfG_{1,1}\big(
\op^{\rm w}(b)\big) = \SfG_{1,1}\big(\op^{\rm w}(b_0)\big),\quad b_0(x, \xi) = 
\frac{1}{4\pi}  \,
\frac{\z(x)}{x}\, \frac{\z(\xi)}{\xi}.
\end{align*}
Since $\p_x\p_\xi b_0(x, \xi) = 0(\lu x\ru^{-2} \lu\xi\ru^{-2})$ and 
$N_{1,1}^\circ[\lu x\ru^{-2} \lu\xi\ru^{-2}] = 0$, we conclude that the sequence 
$\sf w$ defined by \eqref{eq:w} with $a = b_0$, satisfies $M_{1,1}^\circ[{\sf w}] = 0$. 
As a result, Theorem \ref{thm:difference} and Corollary 
\ref{cor:pert} 
guarantee that 
$\SfG_{1,1}\big(\op^{\rm w}(b_0)\big) = \SfG_{1,1}\big(\op^{\rm l}(b_0)\big)$.
The singular values remain the same if we mutliply 
$\op^{\rm l}(b_0)$  
from the right by ${\rm sgn}\, \xi$, and from the left by ${\rm sgn}\, x$, 
which leads to the equality
\begin{align*}
\SfG_{1,1}\big(\op^{\rm l}(b_0)\big) = \SfG_{1,1}\big(\op^{\rm l}(\tilde b_0)\big)
\quad \textup{with}\quad 
\tilde b_0(x, \xi) = \frac{1}{4\pi} \frac{\z(x)\z(\xi)}{|x|\, |\xi|}. 
\end{align*}
By Corollaries \ref{cor:homo} and \ref{cor:pert} again,
\begin{align*}
\SfG_{1,1}\big(\op^{\rm l}(\tilde b_0)\big) 
= \SfG_{1,1}\big(\op^{\rm l}(p_0)\big),\quad 
p_0(x, \xi) = \frac{1}{4\pi} \frac{1}{\lu x\ru\, \lu\xi\ru}.
\end{align*}
To find the functional $\SfG_{1,1}\big(\op^{\rm l}(p_0)\big)$ write
\begin{align*}
n(s; \op^{\rm l}(p_0)) = n(1; \op^{\rm l}(P_0)), 
\end{align*}
where $P_0(x, \xi)$ is the symbol defined by \eqref{eq:tl} with $\l=1$ and 
the potential $s^{-2} V(x)>0$, 
where
\begin{align*}
V(x) = \frac{1}{16\pi^2}\,\frac{1}{1+x^2}.
\end{align*}
By \eqref{eq:bsch1} and \eqref{eq:crit}, 
\begin{align*}
n(s; \op^{\rm l}(p_0)) = \#(1; s^{-2} V)
= \frac{1}{2\pi^2 s} \log \frac{1}{s} + O\bigg(\frac{1}{s}\bigg).
\end{align*}
Substituting this formula in definition \eqref{eq:supinf}, we obtain that 
$\SfG_{1,1}(\op^{\rm l}(p_0)) = (2\pi^2)^{-1}$. Consequently,  
$\SfG_{1,1}(\SA_\Om) = (2\pi^2)^{-1}$ as well. The same argument shows that 
$\sg_{1,1}(\SA_\Om) = (2\pi^2)^{-1}$. This proves \eqref{eq:gs} and hence completes the proof 
of the formula \eqref{eq:mainf}. 
\qed

One should remark that the spectral asymptotics of $\op^{\rm w}(b_0)$ 
could be also derived from \cite[P. 95]{Dauge_Rob1987}, see also 
\cite{Dauge_Rob1986}. In these papers the authors obtain spectral 
asymptotics for a wide class of Weyl-quantized pseudodifferential operators of negative order. 
Given the specific form of the symbol $b_0$, we chose to use a less general 
but more straightforward approach based on the Birman-Schwinger principle.

Sect.  \ref{subsect:straight} together with Sect. \ref{subsect:proofs} 
complete the proof of Theorem \ref{thm:restate},  and hence that of Theorem 
\ref{thm:main}. 
 
\vskip 0.3cm

\textbf{Acknowledgements.} 
The authors are grateful to F. Marceca and G. Rozenblum for bringing to 
their attention references \cite{BonKar2015} and \cite{Dauge_Rob1987}  respectively.

\end{document}